\documentclass[preprint,14pt]{elsarticle}

\usepackage{amsmath,amsfonts,amssymb,amsthm}
\usepackage{hyperref,cleveref}
\usepackage{booktabs,framed,multirow}
\usepackage{graphicx,subfig}
\usepackage{bm}
\usepackage{xcolor}

\usepackage[scale=0.75]{geometry}
\usepackage{algorithm}
\usepackage{algorithmic}
\usepackage{stmaryrd}

\newtheorem{theorem}{Theorem}
\newtheorem{lemma}{Lemma}
\newtheorem{definition}{Definition}
\newtheorem{corollary}{Corollary}
\newtheorem{remark}{Remark}
\newtheorem{example}{Example}

\crefname{figure}{Figure}{Figure}
\crefname{table}{Table}{Table}
\crefname{equation}{Eq.}{Eq.}
\crefname{theorem}{Theorem}{Theorem}
\crefname{lemma}{Lemma}{Lemma}
\crefname{algorithm}{Algorithm}{Algorithm}
\crefname{corollary}{Corollary}{Corollary}
\crefname{remark}{Remark}{Remark}
\crefname{definition}{Definition}{Definition}

\begin{document}
\begin{frontmatter}

\title{A novel parameter-free and  locking-free enriched Galerkin method for linear elasticity }

\author[1]{Shuai Su}
\ead{shuaisu@bjut.edu.cn}

\author[1]{Xiurong Yan}
\ead{yanxr@emails.bjut.edu.cn}

\author[2]{Qian Zhang\corref{cor1}}
\ead{qzhang15@mtu.edu}

\cortext[cor1]{Corresponding author.}

\address[1]{\label{a}  School of Mathematics, Beijing University of Technology, Beijing 100124, China}

\address[2]{Department of Mathematical Sciences, Michigan Technological University, Houghton, MI, 49931, USA}

\begin{abstract}
We propose a novel parameter-free and locking-free enriched Galerkin (EG) method for solving the linear elasticity problem in both two and three dimensions. Unlike existing locking-free EG methods, our method enriches the first-order continuous Galerkin (CG) space with piecewise constants along edges in two dimensions  or faces in three dimensions. This enrichment acts as a correction to the normal component of the CG space, ensuring the locking-free property and delivering an oscillation-free stress approximation without requiring post-processing. Our theoretical analysis establishes the well-posedness of the method and derives optimal error estimates. Numerical experiments further demonstrate the accuracy, efficiency, and robustness of the proposed method.

\end{abstract}


\end{frontmatter}

\section{Introduction}\label{sec1}
Linear elasticity models the deformation and stress distribution in materials or structures under external forces, and is widely applied in structural analysis and engineering design. In this paper, we introduce and analyze a robust and efficient numerical method for solving the linear elasticity problem. 

Various finite element methods have been developed for solving the linear elasticity problem. However, certain widely used elements, such as continuous piecewise linear or bilinear elements \cite{babuvska1992locking}, suffer from a loss of numerical accuracy as the material approaches near-incompressibility, or equivalently, as the Lam\'e constant \( \lambda \to \infty \).  This issue, known as \textit{volumetric locking}, typically manifests in two ways: a severe underestimation of displacement and oscillations in the stress tensor. 
It occurs when the numerical method fails to properly accommodate the incompressibility constraint ($\nabla\cdot\bm u\approx 0$). In theory, the error estimates for displacement and stress exhibit an explicit dependence on $\lambda$.

To circumvent locking, various approaches have been proposed based on the primal formulation, including \( C^0 \) finite elements with reduced integration \cite{brezzi2012mixed}, high-order \( C^0 \) elements \cite{babuvska1992locking}, and nonconforming finite elements \cite{brenner1992linear,falk1991nonconforming}. It is worth noting that these methods are locking-free in the sense of Babu\v{s}ka and Suri \cite{babuvska1992locking} -- that is, they primarily address locking in terms of displacement and energy. In the nearly incompressible case, additional post-processing techniques are required to obtain an accurate approximation of hydrostatic stress $\frac1d\operatorname{tr}{\bm\sigma}$ or pressure $\lambda \nabla\cdot\bm u$.
To address this issue, one approach is to use the divergence-free Stokes velocity elements \cite{BURMAN2020113224} in the primal formulation. However, constructing such elements poses significant challenges, especially in three dimensions (3D). Consequently, the resulting elements have either a large number of degrees of freedom (DoFs) or complex macro-element structures \cite{neilan2015discrete,fu2020exact,guzman2022exact,hu2022family}. 
Another approach is to apply the mixed formulation with stress tensor $\bm\sigma$ \cite{brezzi2012mixed}. Yet, it requires a careful selection of inf-sup stable finite element pairs that also preserve stress symmetry \cite{johnson1978some,arnold2002mixed,arnold2008finite,hu2015family}. In addition, introducing the stress tensor in the formulation increases the computational costs.
Alternatively, discontinuous Galerkin (DG) methods, see, e.g., \cite{cockburn2006discontinuous,wang2020mixed} and weak Galerkin (WG) methods, see, e.g., \cite{yi2019lowest,chen2016robust,huo2024locking} also address locking issue. Nonetheless, these methods typically require more DoFs to achieve the same convergence rates due to the use of piecewise polynomials. 

Recently, the enriched Galerkin (EG) method has attracted considerable attention due to its high efficiency. It was originally introduced by Sun and Liu \cite{sun2009locally} for solving second-order elliptic problems and shown to be locally mass conservative. Its key idea is to enrich the continuous Galerkin (CG) finite element space with a DG space and incorporate it into a DG formulation. The EG method retains the advantageous properties of the DG method while maintaining a computational cost comparable to that of the CG method.
To date, the EG method has been successfully applied to a wide range of problems, including elliptic and parabolic equations in porous media \cite{lee2016locally}, two-phase flow simulations \cite{hu2024pressure,lee2018enriched}, and the Stokes problem \cite{chaabane2018stable,yi2022enriched}, etc. In particular, 
a first-order locking-free EG method based on the primal formulation was proposed in \cite{yi2022locking} for solving linear elasticity problem, and later, a parameter-free and arbitrary-order extension was introduced in \cite{su2024parameter}. While extensive numerical experiments have confirmed the locking-free property of the latter EG method in \cite{su2024parameter}, establishing a rigorous theoretical foundation remains challenging.
In addition, both methods in \cite{yi2022locking,su2024parameter} require a post-processing technique to achieve oscillation-free stress approximation. 

To address this, we introduce a novel parameter-free and locking-free EG (PF\&LF-EG) method for  linear elasticity problems in both two dimensions (2D) and 3D.  Unlike traditional EG methods that enrich the CG space with a DG space defined on elements, our approach enriches the first-order CG space with a DG space of piecewise constants on edges in  2D  or faces in 3D. This EG space enables us to establish a commutativity property between projections and the weak divergence operator. Consequently, we derive rigorous error estimates that remain independent of $\lambda$, providing a theoretical guarantee of the method's locking-free property.
To achieve a parameter-free formulation, we define a weak gradient and weak divergence by incorporating the enriched component and replace the standard gradient and divergence operators in the weak formulation with their weak counterparts. To ensure well-posedness, we introduce a stabilization term similar to that used in the WG method.
The PF\&LF-EG method has four key features:
\begin{itemize}
\item It is theoretically proven to be locking-free and achieves accurate stress approximation without requiring post-processing techniques.
\item Compared to mixed methods, WG methods, and DG methods, it significantly reduces the DoFs, enhancing computational efficiency.
\item It produces a continuous displacement approximation. 
\item It is parameter-free, eliminating the need for additional parameter tuning.
\end{itemize}




The remainder of this paper is organized as follows. 
Section \ref{sec:pre} introduces some notations and the model problem.  
In Section \ref{sec:EG}, we develop our novel PF\&LF-EG method for linear elasticity. In Section \ref{sec:Theore}, we establish the theoretical analysis of the  method including the well-posedness and error estimates. Several numerical experiments are conducted  in Section \ref{sec:num} to validate the effectiveness of the method and some concluding remarks are given in Section \ref{conclusion}.

\section{Preliminaries} \label{sec:pre}
\subsection{Notations}
Throughout this paper,
we denote by \( H^s(D) \) the Sobolev space  on a bounded Lipschitz domain \( D \subset \mathbb{R}^d \)  with the associated norm \( \|\cdot\|_{s,D} \), where \( d =2, 3 \) and \( s \geq 0 \). The space \( H^0(D) \) coincides with \( L^2(D) \) where the \(L^2\)-inner product  is denoted by \( (\cdot, \cdot)_D \). For simplicity, we omit the subscript \( D \) or \( s \)  when \( D = \Omega \) or \( s = 0 \), provided that no ambiguity arises. We also denote by  \( P_{\ell}(D) \) the space of  ${\ell}$-th order polynomials on  \( D \). These notations extend naturally to vector- and tensor-valued Sobolev spaces. 
We denote by $\big[P_0(D)\big]^{d\times d}_{\operatorname{sym}}$ the subspace of $\big[P_0(D)\big]^{d\times d}$ consisting of all symmetric matrices. 

Let \( \mathcal{T}_h \) be a quasi-uniform mesh partitioning the domain \( \Omega \) into triangles in 2D or tetrahedra in 3D. Denote by \( \mathcal{E}_h \) the set of all edges in 2D or faces in 3D associated with \( \mathcal{T}_h \). For each element \( T \in \mathcal{T}_h \), let \( h_T \) be its diameter, and define the mesh size as  $h = \max_{T \in \mathcal{T}_h} h_T$.
The space of piecewise polynomials of degree at most \( \ell \) over \( \mathcal{T}_h \) is denoted by \( P_{\ell}(\mathcal{T}_h) \). 

We use \( C \)  to denote a generic positive constant, independent of the mesh size \( h \), though its value may vary in different estimates.

\subsection{The model problem}
We consider the linear elasticity problem on a connected, bounded, and Lipschitz domain $\Omega \subset \mathbb{R}^d(d=2,3)$. The problem finds the displacement vector $\bm u$ such that
\begin{subequations}\label{eq:l1}
\begin{align}
-\nabla \cdot \bm{\sigma}(\bm{u}) & =\boldsymbol{f} & & \text { in } \Omega, \label{eq:l1a}\\
\bm{u} & =\bm 0 & & \text { on } \partial \Omega, \label{eq:l1b}
\end{align}
\end{subequations}
where $\bm{f}$ is the body force and $\bm{\sigma}$ is the symmetric $d \times d$ stress tensor defined by
$$
\bm{\sigma}(\bm{u}):=2 \mu \bm{\epsilon}(\bm{u})+\lambda(\nabla \cdot \bm{u}) \mathbf{I}.
$$ 
Here, $\bm{\epsilon}(\bm{u}) = \frac{1}{2}(\nabla \bm{u} + (\nabla \bm{u})^\mathrm{T})$ is the strain tensor, $\mathbf{I}$ is the $d \times d$ identity matrix, and $\lambda$,  $\mu$ are the Lam\'e parameters. The parameters $\lambda$,  $\mu$  satisfy $0 < \lambda < \infty$ and $0 < \mu_1 < \mu < \mu_2$ for some positive constants $\mu_1$ and $\mu_2$.  
For plane strain in 2D and for 3D, the Lam\'e parameters $\lambda$ and $\mu$ can be expressed in terms of the Young's modulus $E$ and the Poisson's ratio $\nu$ as  
\[
\lambda = \frac{E \nu}{(1+\nu)(1-2\nu)}, \quad \mu = \frac{E}{2(1+\nu)}.
\]
The primal formulation of \eqref{eq:l1} is to seek $\bm{u} \in\left[H_0^1(\Omega)\right]^d$ such that 
\begin{align}\label{primal-form}
    2 \mu(\bm{\epsilon}(\bm{u}), \bm{\epsilon}(\bm{v}))+\lambda(\nabla \cdot \bm{u}, \nabla \cdot \bm{v})=(\bm{f}, \bm{v}) \quad
\hbox{ for any }  \bm{v} \in\left[H_{0}^1(\Omega)\right]^d.
\end{align}
Assume that the solution \( \bm u \) satisfies the \( H^2 \)-regularity estimate \cite{brenner1992linear,susanne1994mathematical,grisvard1992singularities}  
\begin{align}\label{regularity}
  \|\bm u\|_2+\lambda\|\nabla \cdot \bm u\|_1 \leq C\|\bm f\|.
\end{align}

\section{A novel  parameter-free and locking-free enriched Galerkin method}\label{sec:EG}
In this section, we introduce the PF\&LF-EG method for solving the linear elasticity problem.

We first define the finite element space. Recall the vector-valued linear CG finite element space
\begin{equation*}
 \mathrm{CG}=\left\{\bm {v} \in [H^1(\Omega)]^d:\left.\bm {v}\right|_T \in\left[P_1(T)\right]^d \hbox{ for all }  T \in \mathcal{T}_h\right\}.
\end{equation*}
The EG space is obtained by enriching the CG space with an appropriate DG space. 
As noted in \cite[Remark 1]{su2024parameter}, for the EG method to be locking-free, it is essential that the interpolation operator and the (weak) divergence operator satisfy a commutativity property. Moreover, this property plays a crucial role in achieving an oscillation-free approximation of the stress tensor. However, proving the commutativity property becomes challenging when the CG space is enriched solely with piecewise polynomials defined on elements as the EG method in \cite{su2024parameter}.
To address this issue, in this paper, we define a DG space which contains piecewise constant functions over $\mathcal{E}_h$:
\begin{equation*}
\mathrm{DG}=\left\{v \in L^2\left(\mathcal{E}_h\right):\left.v\right|_e \in P_0(e) \hbox{ for all }   e \in \mathcal{E}_h\right\}.
\end{equation*}
Now we introduce the EG space $\bm V_h$  for the  
displacement $\bm u$ as
$$
\bm V_h :=\left\{\bm {v}_h=\left\{\bm{v}_0, v_b\right\}: \bm{v}_0 \in \mathrm{CG} \text { and } v_b \in \mathrm{DG}\right\}.
$$
Here we note that the component $v_b|_e$ serves as a correction to $\frac{1}{|e|} \int_e \bm{v}_0 \cdot \bm{n}_e \mathrm{~d} s$, where $\bm {n}_e$ represents the assigned unit normal vector to the edge $e$.
Compared with the discrete spaces in the  hybrid high-order  \cite{di2014arbitrary}
and WG methods \cite{wang2013weak}, the key distinction is the usage of the
CG space instead of DG space for $\bm {u}_0$. 

To realize parameter-free property, for any function \( \bm{v} = \{\bm{v}_0, v_b\} \in \bm{V}_h \), we  define a weak gradient and a weak divergence by substituting the normal component in the integration by parts formula with the enriched component \( v_b \).

\begin{definition}\label{def1}
For $\bm{v} \in \bm V_h$, the weak gradient  operator is defined as  $\nabla_{w} \bm{v} \in \left[P_0(\mathcal T_h)\right]^{d \times d}$ satisfying
\begin{equation}\label{eq:grad}
\left(\nabla_{w} \bm{v}, \bm\tau\right)_T = \left\langle v_b \bm{n}_e \cdot \bm{n}, \bm{n} \cdot \bm\tau \cdot \bm{n} \right\rangle_{\partial T} + \left\langle \bm{n} \times \bm{v}_0, \bm{n} \times \bm\tau \cdot \bm{n} \right\rangle_{\partial T},\quad \forall \bm\tau \in \left[P_0(T)\right]^{d \times d},\ T \in \mathcal T_h,
\end{equation}
and the weak divergence operator is defined 
as $\nabla_{w} \cdot \bm{v} \in P_0(\mathcal T_h)$ satisfying
\begin{equation}\label{eq:div}
\left(\nabla_{w} \cdot \bm{v}, \varphi\right)_T = \left\langle v_b \bm{n}_e \cdot \bm{n}, \varphi \right\rangle_{\partial T},\quad \forall \varphi \in P_0(T),\ T \in \mathcal T_h,
\end{equation}
where $\bm n$ represents the unit outward normal vector to $\partial T$ and ${\bm n}_e$ represents the assigned unit normal vector to the edge or face $e \subset \partial T$.
\end{definition}
\noindent We then denote
\begin{align}\label{weakeps}
\bm{\epsilon}_w(\bm{v})=\frac{1}{2}\left(\nabla_w \bm{v}+(\nabla_w\bm{v})^{\mathrm{T}}\right)\text{ and }
\bm{\sigma}_w(\bm{v}):=2 \mu \bm{\epsilon}_w(\bm{v})+\lambda(\nabla_w \cdot \bm{v}) \mathbf{I}.
\end{align}

Let $\bm Q_0$ denote the Scott-Zhang type interpolation operator \cite{scottzhang} from $\left[L^2(\Omega)\right]^d$ onto the space of CG and $Q_b$ the $L^2$ projection from $L^2(\mathcal E_h)$ onto the space $\operatorname{DG}$.
For any $\bm v\in\left[L^2(\Omega)\right]^d$, define $v_n\in L^2(\mathcal E_h)$ by
\begin{align}
    v_n|_e = \bm v|_e\cdot \bm n_e \text{ for any }e\in\mathcal E_h.\label{vdotn}
\end{align}
We then define $\bm Q_h\bm v$ as
\[\bm Q_h\bm v =\{\bm Q_0\bm v,Q_b v_n\}.\]

According to  Definition \ref{def1} and \eqref{weakeps}, for $\bm v=\{\bm v_0,v_b\}\in\bm V_h$, we have
\begin{align}
    \left(\bm{\epsilon}_w (\bm{v}), \bm\tau\right)_T &= (\bm{\epsilon}(\bm{v}_0), \bm \tau)_T- \big\langle (Q_bv_{0,n}-v_b) \bm n_e \cdot \bm n , \bm n \cdot \bm\tau \cdot \bm n \big\rangle_{\partial T}, \quad \forall \bm\tau \in \left[P_0(T)\right]^{d \times d}_{\operatorname{sym}},\label{diff-eps}\\
\left(\nabla_{w} \cdot \bm{v}, \varphi\right)_T & = \left(\nabla \cdot \bm{v}_0, \varphi\right)_T -\left\langle (Q_bv_{0,n}-v_b) \bm{n}_e \cdot \bm{n}, \varphi \right\rangle_{\partial T},\quad \forall \varphi \in P_0(T),\label{diff-div}
\end{align}
where $v_{0,n}$ follows the definition given in \eqref{vdotn}.

We now define the EG space with vanishing Dirichlet boundary condition as
$$
\bm V_h^0=\left\{{\bm v}_h=\left\{{\bm v}_0, v_b\right\} \in \bm V_h: \ {\bm v}_0=0,\ v_b=0 \text { on } \partial \Omega\right\}.
$$
Then we present the PF\&LF-EG method for the linear elasticity problem with  vanishing Dirichlet boundary condition as follows.
\begin{algorithm}
\caption{A  PF\&LF-EG method for vanishing Dirichlet boundary condition}\label{eg}
\begin{algorithmic}\STATE
Find $\bm u_h=\left\{\bm u_0,u_b\right\} \in \bm V_h^0$ such that
\begin{align}\label{eq:al1}
\bm a\left(\bm u_h, \bm v\right)=(\bm f, \bm v_0),\quad \forall \bm v=\{\bm v_0,v_b\} \in \bm V_h^{0}, 
\end{align}
where
$$
\begin{aligned}
\bm a (\bm w, \bm v) & =2 \mu \sum_{T \in \mathcal{T}_h}\left(\bm \epsilon_w(\bm w), \bm \epsilon_w(\bm v)\right)_T +\lambda \sum_{T \in \mathcal{T}_h}\left(\nabla_w \cdot \bm w, \nabla_w \cdot \bm v\right)_T+s(\bm w, \bm v),
\end{aligned}
$$
with $s(\bm w, \bm v)= \sum_{T \in \mathcal{T}_h} h_T^{-1}\left\langle Q_b w_{0,n} -w_b, Q_b v_{0,n}-v_b\right\rangle_{\partial T}$. 
\end{algorithmic}
\end{algorithm}
\begin{remark}
We also consider the mixed boundary conditions
$$
\bm{u}=\bm{u}_{D} \  \text { on } \Gamma_D\ \text{ and 
 }\ \bm{\sigma}(\bm{u}) \bm{n} =\bm{g}  
 \ \text { on } \Gamma_N,
$$
where $\Gamma_D$ and $\Gamma_N$ are the Dirichlet and Neumann boundaries satisfying $\partial \Omega=\Gamma_D \cup \Gamma_N$ and $\Gamma_D \cap \Gamma_N=\emptyset$, $\bm{g}$ is the traction posed to $\Gamma_N$, and $\bm{u}_D$ is a given function. We define the EG space with boundary conditions
$$
\bm{V}_h^{0, D}=\left\{\bm{v}_h=\left\{\bm{v}_0, v_b\right\} \in \bm{V}_h: \bm{v}_0=0, v_b=0 \text { on } \Gamma_D\right\}.
$$
The PF$\&$LF-EG method is to find $\bm u_h=\left\{\bm u_0,u_b\right\} \in \bm V_h$ with $u_b = Q_b u_{D,n}$ and $ \bm u_0=
\Pi_h\bm u_D$ on $\Gamma_D$ such that
\begin{align}\label{eq:al2}
\bm a\left(\bm u_h, \bm v\right)=\bm F(\bm v),\quad \forall \bm v \in \bm V_h^{0,D},
\end{align}
where
$$
\begin{aligned}
\bm F(\bm v) & =(\bm f, \bm v_0)+\langle \bm g\cdot \bm n_e, v_b\rangle_{\Gamma_N}+\langle\bm n \times \bm g , \bm n \times \bm v_0\rangle_{\Gamma_N}
\end{aligned}
$$
and $\Pi_h$ represents the first-order Lagrange interpolation.

For the pure traction problem $(\Gamma_D=\emptyset)$, we assume the following compatibility condition to hold
\[(\bm f,\bm v) + \langle \bm g,\bm v\rangle_{\partial\Omega}=0, \quad\forall \bm v\in \operatorname{RM},\]
where $\operatorname{RM}=\{\bm a + b\bm x^{\perp}: \bm a\in \mathbb R^{2},b\in\mathbb R\}$ with $\bm x^{\perp}=(x_2,-x_1)^{\mathrm T}$ in 2D and $\operatorname{RM}=\{\bm a + \bm b\times\bm x:\bm a,\bm b\in \mathbb R^{3}\}$  in 3D. In this case, we seek the numerical solution $\bm u_h\in \widehat{\bm V}_h$ which is defined by
    \[\widehat{\bm V}_h=\{\bm u_h=\{\bm u_0,u_b\}\in \bm V_h:\int_{\Omega}\bm u_0 = 0\text{ and }\int_{\Omega}\nabla\times\bm u_0 = 0\}.\]
\end{remark}


\section{Theoretical analysis}\label{sec:Theore}
In this section, we perform the theoretical analysis for the proposed
PF\&LF-EG method in Algorithm \ref{eg}. 
Let $\mathcal Q_h$ and $\mathbb Q_h$ be two $L^2$ projections onto $P_{0}(\mathcal T_h)$ and $\left[P_{0}(\mathcal T_h)\right]^{d \times d}$, respectively.

\subsection{Well-posedness}
For $\bm v \in \bm V_h^0$, we define
\begin{equation}\label{md-norm}
|\!|\!|\bm v |\!|\!|^2=\sum_{T \in \mathcal{T}_h}\left\|\bm{\epsilon}_w (\bm v) \right\|_T^2+\sum_{T \in \mathcal{T}_h} h_T^{-1}\left\|Q_b v_{0,n} - v_b\right\|_{\partial T}^2,
\end{equation}
\begin{equation}\label{norm}
\|\bm v\|_{1, h}^2=\sum_{T \in \mathcal{T}_h}\|\bm{\epsilon}(\bm v_0)\|_T^2+\sum_{T \in \mathcal{T}_h} h_T^{-1}\left\|Q_b v_{0,n} - v_b\right\|_{\partial T}^2.
\end{equation}
According to the Korn's inequality \cite{genKorn}, for $\bm v_0\in \operatorname{CG}$ with $\bm v_0 =0 $ on $\partial\Omega$, there exists a constant $C>0$ such that
\begin{equation}
\label{generalizedkorn}
   \|\nabla\boldsymbol{v}_0\|\leq
C\|\boldsymbol{\epsilon}(\boldsymbol{v}_0)\|.
\end{equation}
Assume $\bm v\in\bm V_h^0$ satisfies $\|\bm v\|_{1, h}=0$.  From \eqref{generalizedkorn}, we deduce $\bm v_0=0$. Consequently \[\|\bm v\|^2_{1, h}=\sum_{T \in \mathcal{T}_h} h_T^{-1}\left\|Q_b v_{0,n} - v_b\right\|_{\partial T}^2=\sum_{T \in \mathcal{T}_h} h_T^{-1}\left\| v_b\right\|_{\partial T}^2=0,\] which yields $v_b=0$. Therefore,  $\|\cdot\|_{1, h}$ defines a norm in $\bm V_h^0$. 
To show that $|\!|\!|\cdot |\!|\!|$ also defines a norm in $\bm V_h^0$, we present the following lemma.
\begin{lemma}\label{lem:2}
For any $\bm v \in \bm V_h^0$, there are positive constants $C_1$ and $C_2$ independent of $h$ such that
\begin{equation}\label{norm equi}
C_1 \|\bm v\|_{1, h} \leq |\!|\!|\bm v |\!|\!| \leq C_2 \|\bm v\|_{1, h} .
\end{equation}
\end{lemma}

\begin{proof}
For $\bm\tau\in [P_0(\mathcal T_h)]_{\operatorname{sym}}^{d\times d} $, it follows from \eqref{diff-eps} that
$$
\begin{aligned}
(\bm \epsilon (\bm v_0),\bm\tau)_{T}-(\bm \epsilon_w (\bm v ),\bm\tau)_{T}= \big\langle (Q_bv_{0,n}-v_b) \bm n_e \cdot \bm n , \bm n \cdot \bm\tau \cdot \bm n \big\rangle_{\partial T}.\\
\end{aligned}
$$
By the definition of $Q_b$ and the trace inequality, we have
\begin{align}
(\bm \epsilon (\bm v_0),\bm\tau)_{T}-(\bm \epsilon_w (\bm v ),\bm\tau)_{T}\leq Ch_T^{-\frac{1}{2}}\|Q_b v_{0,n}-v_b\|_{\partial T}\|\bm\tau\|_{T}.\label{lemma2-1}
\end{align}
Letting $\bm\tau=\bm \epsilon (\bm v_0)$ in \eqref{lemma2-1} and summing up over $T\in\mathcal T_h$, we obtain
\begin{align}\label{ineq1}
    \Big(\sum_{T\in\mathcal T_h}\|\bm \epsilon (\bm v_0)\|_{T}^2\Big)^{1/2} \leq C |\!|\!| \bm v |\!|\!|.
\end{align}
Similarly, letting $\bm\tau=-\bm \epsilon_w (\bm v_0)$ in \eqref{lemma2-1}, we obtain
\begin{align}\label{ineq2}
\Big(\sum_{T\in\mathcal T_h}\|\bm \epsilon_w (\bm v)\|_{T}^2\Big)^{1/2}\leq C \|\bm v\|_{1,h}.
\end{align}
The above two inequalities \eqref{ineq1} and \eqref{ineq2} lead to \eqref{norm equi}.
\end{proof}

\begin{theorem}
The Algorithm \ref{eg} has a unique solution.    
\end{theorem}

\begin{proof}
 We only need to demonstrate that $\bm u_h=0$ is the unique solution of \eqref{eq:al1} when $\bm f=0$. Since $|\!|\!|\cdot|\!|\!|$ defines a norm, $(\bm f,  \bm u_0)=0$ leads to $|\!|\!| \bm u_h |\!|\!| =0$, which implies $\bm u_h=0$.
\end{proof}
\subsection{Error equations}
In this subsection, we derive error equations for discrete displacement, in preparation for the convergence analysis in the subsequent subsection. 
Let $\bm{u}_h=\left\{\bm{u}_0, u_b\right\} \in \bm{V}_h^0$ be the discrete solution of the PF\&LF-EG method \eqref{eq:al1} and $\bm{u}$ be  the exact solution of the linear elasticity problem \eqref{eq:l1}. 
As previously defined, the projection of $\bm{u}$ to the finite element space $\bm{V}_h$ is given by $
\bm{Q}_h \bm{u}=\left\{\bm{Q}_0 \bm{u}, Q_b u_n\right\}.$
The error $\bm{e}_h$ is defined as
$$
\bm{e}_h=\left\{\bm{e}_0, e_b\right\}=\bm{Q}_h \bm{u}-\bm u_h=\left\{\bm{Q}_0 \bm{u}-\bm{u}_0, Q_b u_n-u_b\right\} .
$$
We first show the following (quasi-)commutativity properties.
\begin{lemma}
For $\bm{v} \in\left[H^1(\Omega)\right]^d$ and $\bm{w} \in H(\operatorname{div}; \Omega)$, the interpolation or projection operators $\bm Q_h, \mathbb {Q}_h$, and $\mathcal Q_h$ satisfy the following (quasi-)commutativity properties 
\begin{align}\label{eq:Qhg}
\left(\bm \epsilon_w\left(\bm Q_h \bm{v}\right), \bm\tau_h \right)_T& =\left(\mathbb Q_h\bm \epsilon( \bm{v}), \bm\tau_h\right)_T+\left\langle \bm{n} \times (\bm Q_0 \bm{v}  - \bm{v} ),\bm{n} \times  \bm\tau_h \cdot \bm{n} \right\rangle_{\partial T}\quad  \forall \bm\tau_h \in\left[P_{0}(\mathcal T_h)\right]^{d \times d}_{\operatorname{sym}},\\
\left(\nabla_w \cdot\left(\bm Q_h \bm{w}\right),\varphi_h \right)_T & =\left(\mathcal Q_h(\nabla \cdot \bm{w}),\varphi_h \right)_T\quad\forall \varphi_h \in P_{0}(\mathcal T_h).\label{eq:Qhd}
\end{align}
\end{lemma}
\begin{proof} 
For any $\bm\tau_h  \in\left[P_{0}(\mathcal T_h)\right]^{d\times d}_{\operatorname{sym}}$, it follows from  \eqref{eq:grad}, the definitions of $\bm Q_h$ and $\mathbb Q_h$, as well as integration by parts, that
$$
\begin{aligned}
\left(\nabla_w\left(\bm Q_h \bm v\right),\bm\tau_h \right)_T& =\left\langle Q_b v_n\bm n_e\cdot\bm n , \bm{n} \cdot \bm\tau_h  \cdot \bm{n}\right\rangle_{\partial T}+\left\langle\bm{n}\times \bm Q_0 \bm v ,\bm{n} \times \bm\tau_h \cdot \bm{n}\right\rangle_{\partial T}\\
& =\left\langle \bm v \cdot \bm{n}, \bm{n} \cdot \bm\tau_h  \cdot \bm{n}\right\rangle_{\partial T}+\left\langle \bm{n} \times \bm Q_0 \bm v ,\bm{n} \times \bm\tau_h  \cdot \bm{n} \right\rangle_{\partial T}\\
&\ \ \  +\langle \bm{n} \times \bm v, \bm{n} \times \bm\tau_h \cdot \bm{n}\rangle_{\partial T}-\langle \bm{n} \times \bm v, \bm{n} \times \bm\tau_h  \cdot \bm{n}\rangle_{\partial T} \\
& =\left(\nabla \bm v, \bm\tau_h \right)_T +\left\langle \bm{n} \times ( \bm Q_0 \bm v  - \bm v ), \bm{n} \times \bm\tau_h  \cdot \bm{n} \right\rangle_{\partial T}\\
& =\left(\mathbb Q_h(\nabla \bm v), \bm\tau_h \right)_T+\left\langle \bm{n} \times ( \bm Q_0 \bm v  - \bm v), \bm{n} \times \bm\tau_h  \cdot \bm{n} \right\rangle_{\partial T},
\end{aligned}
$$
which leads to \eqref{eq:Qhg}.

Similarly, for any $\varphi_h \in P_{0}(\mathcal T_h)$, applying \eqref{eq:div}, the definitions of $Q_b$ and $\mathcal Q_h$, and integration by parts, we derive
\begin{align*}
\left(\nabla_w \cdot\left(\bm Q_h \bm w\right), \varphi_h \right)_T=\left\langle Q_bw_n\bm n_e\cdot\bm n, \varphi_h \right\rangle_{\partial T}
=\langle\bm w \cdot \bm{n}, \varphi_h \rangle_{\partial T}  =(\nabla \cdot \bm w, \varphi_h)_T =\left(\mathcal Q_h(\nabla \cdot \bm w), \varphi_h \right)_T .
\end{align*}
This completes the proof. 
\end{proof}


\begin{lemma}
 For any $\bm v \in \bm V_h^0$, the following error equation holds,
\begin{align}\label{erreq}
    \bm a\left(\bm e_h, \bm{v}\right)=2 \mu \ell_1 (\bm u, \bm v)+2 \mu\ell_2(\bm u, \bm v)+ \lambda\ell_3(\bm{u}, \bm{v})+ s\left(\bm Q_h \bm{u}, \bm{v}\right),
\end{align}
where
$$
\begin{aligned}
\ell_1 (\bm u, \bm v)& =\sum_{T \in \mathcal{T}_h}\big\langle(Q_bv_{0,n}-v_b)\bm{n}_e \cdot \bm{n}, \bm{n} \cdot \left(\bm \epsilon (\bm u)-\mathbb{Q}_h\bm \epsilon (\bm u) \right)\cdot \bm{n}\big\rangle_{\partial T}, \\
\ell_2(\bm u, \bm v)& = \sum_{T \in \mathcal{T}_h}\big\langle \bm{n} \times (\bm Q_0 \bm u  - \bm u ) , \bm{n} \times \bm \epsilon_w (\bm v ) \cdot \boldsymbol{n}\big\rangle_{\partial T},\\
\ell_3(\bm u , \bm v)&  =\sum_{T \in \mathcal{T}_h}\big\langle(Q_bv_{0,n}-v_b) \bm{n}_e \cdot \bm{n}, \nabla\cdot\bm u-\mathcal Q_h\nabla\cdot\bm u\big\rangle_{\partial T}.
\end{aligned}
$$
\end{lemma}

\begin{proof}
We begin by subtracting \( \bm a\left(\bm u_h, \bm v\right) \) from \( \bm a\left(\bm Q_h \bm u, \bm v\right) \) and applying integration by parts to obtain
\begin{equation}\label{aqhuh}
\begin{aligned}
\bm a\left(\bm Q_h \bm u-\bm u_h, \bm v\right) = &\ \bm a\left(\bm Q_h \bm u, \bm v\right)-(\bm f, \bm v_0)=\bm a\left(\bm Q_h \bm u, \bm v\right)+(\nabla\cdot\bm\sigma(\bm u), \bm v_0)\\
 =&  \ 2 \mu\sum_{T \in \mathcal{T}_h} \Big(\big(\bm \epsilon_w(\bm Q_h \bm u), \bm \epsilon_w( \bm v) \big)_T-\big(\bm \epsilon(\bm u), \bm \epsilon( \bm v_0) \big)_T \Big)+ s(\bm Q_h \bm u, \bm v)\\
& +\lambda \sum_{T \in \mathcal{T}_h}\Big(\left(\nabla_w \cdot (\bm Q_h \bm u), \nabla_w \cdot \bm v\right)_T - \big(\nabla \cdot \bm u , \nabla  \cdot \bm v_0 \big)\Big).
\end{aligned}
\end{equation}
To further simplify \eqref{aqhuh}, we use  \eqref{eq:Qhg} and \eqref{diff-eps} to derive
\begin{equation}\label{eqn-11}
\begin{aligned}
\left(\bm\epsilon_w\left(\bm Q_h \bm u\right), \bm \epsilon_w(\bm v)\right)_T  
=&\ \left(\mathbb Q_h\bm \epsilon( \bm u), \bm\epsilon_w (\bm v)\right)_T +\left\langle \bm{n} \times (\bm Q_0 \bm u  - \bm u ),\bm{n} \times \bm\epsilon_w (\bm v) \cdot \bm{n}\right\rangle_{\partial T}\\
 =&\ \left(\bm\epsilon (\bm u), \bm\epsilon( \bm v_0)\right)_T-\left\langle (Q_bv_{0,n}-v_b)\bm{n}_e \cdot \bm{n},\bm{n} \cdot \mathbb Q_h\bm\epsilon (\bm u) \cdot \bm{n}\right\rangle_{\partial T}\\
& + \left\langle \bm{n} \times (\bm Q_0 \bm u  - \bm u) , \bm{n} \times \bm\epsilon_w \bm v \cdot \bm{n}\right\rangle_{\partial T}.
\end{aligned}
\end{equation}
Similarly, using \eqref{eq:Qhd} and \eqref{diff-div}, we derive
\begin{equation}\label{eqn-div}
\begin{aligned}
\left(\nabla_w\cdot\left(\bm Q_h \bm u\right), \nabla_w\cdot \bm v\right)_T  =&\ \left(\mathcal Q_h(\nabla \cdot\bm u), \nabla_w \cdot\bm v\right)_T\\
  =&\ \left(\nabla\cdot \bm u,\nabla \cdot\bm v_0\right)_T-\left\langle (Q_bv_{0,n}-v_b)\bm{n}_e \cdot \bm{n}, \mathcal Q_h(\nabla \cdot\bm u)\right\rangle_{\partial T}.
\end{aligned}
\end{equation}
We finally obtain \eqref{erreq} by substituting \eqref{eqn-11} and \eqref{eqn-div} into \eqref{aqhuh}.
\end{proof}
\subsection{Error estimates}

In this subsection, we provide estimates for the terms on the right-hand side of the error equation \eqref{erreq} and derive error bounds for both the displacement and the stress tensor.
\begin{lemma}[\cite{scottzhang,monk2003}]
Assume that $\bm{w} \in\left[H^{2}(\Omega)\right]^d$, $\bm{\tau} \in\left[H^{2}(\Omega)\right]^{d\times d}$, and $\rho \in H^1(\Omega)$. Then we have
\begin{align}\label{eqn-7.1}
&\sum_{T \in \mathcal{T}_h}\|\bm{w}-\bm Q_0 \bm{w}\|_{m,T}^2 \leq Ch^{2(l-m)}\|\bm{w}\|_{l}^2\quad  \text{for }  0 \leq m \leq l \leq 2,\\
& \sum_{T \in \mathcal{T}_h} \left\|\bm{\tau}-\mathbb Q_h \bm{\tau}\right\|^2_T \leq C h^{2 }\|\bm{\tau}\|_{1}^2,
\label{eqn-7.2}\\
& \sum_{T \in \mathcal{T}_h}\left\|\rho-\mathcal Q_h \rho\right\|_T^2 \leq C h^{2}\|\rho\|_1^2 .\label{eqn-7.3}
\end{align}
\end{lemma}

\begin{lemma}\label{le:erres}
Assume that $\bm u \in\left[H^{2}(\Omega)\right]^d$ and $\bm v \in \bm V_h^0$. Then, we have
\begin{align*}
 \left|\ell_1(\bm u, \bm v )\right| & \leq C h \|\bm u\|_2 |\!|\!|\bm v |\!|\!|, \\
 \left|\ell_2(\bm u, \bm v )\right| & \leq C h \|\bm u\|_2 |\!|\!|\bm v |\!|\!|, \\
 \left|\ell_3(\bm u, \bm v )\right| & \leq C h \|\nabla\cdot\bm u\|_1 |\!|\!|\bm v |\!|\!|, \\
 \left|s(Q_h \bm u, \bm v  )\right| & \leq C h \|\bm u\|_2 |\!|\!|\bm v |\!|\!|.
\end{align*}
\end{lemma}

\begin{proof}
To estimate \( \ell_1(\bm u, \bm v) \), we apply the Cauchy-Schwarz inequality, the trace inequality, and \eqref{eqn-7.2} to obtain  
$$
\begin{aligned}
\big|\ell_1(\bm u, \bm v)\big|& =\bigg|\sum_{T \in \mathcal{T}_h}\big\langle(Q_bv_{0,n}-v_b) \bm{n}_e \cdot \bm{n}, \bm{n} \cdot \left(\bm \epsilon (\bm u)-\mathbb{Q}_h\bm \epsilon (\bm u) \right)\cdot \bm{n}\big\rangle_{\partial T}\bigg| \\
& \leq C \bigg(\sum_{T \in \mathcal{T}_h} h_T\|\bm \epsilon (\bm u)-\mathbb{Q}_h\bm \epsilon (\bm u) \|_{\partial T}^2\bigg)^{1/2}\bigg(\sum_{T \in \mathcal{T}_h} h_T^{-1}\|Q_bv_{0,n}-v_b\|_{\partial T}^2\bigg)^{1/2}\\
& \leq C h \|\bm u\|_2  |\!|\!| \bm v|\!|\!|.
\end{aligned}
$$
Similarly, we apply \eqref{eqn-7.1} and \eqref{eqn-7.3} to $\ell_2(\bm u, \bm v)$ and $\ell_3(\bm u, \bm v)$, respectively,  to derive
\begin{align*}
\big| \ell_2(\bm u, \bm v) \big|
\leq &  \ C h \|\bm u\|_{2}|\!|\!|\bm v|\!|\!|,\\
\big|\ell_3(\bm u, \bm v)\big|  \leq  &\ C h\|\nabla\cdot\bm u\|_1|\!|\!| \bm v|\!|\!|.
\end{align*}
Finally, we estimate \( s(Q_h \bm u, \bm v) \) using the definition of \( Q_b \), the trace inequality, and \eqref{eqn-7.1}
\begin{align}
\big|s\left(\bm Q_h \bm u, \bm v\right)\big| & = \bigg|\sum_{T \in \mathcal{T}_h} h_T^{-1}\big\langle Q_b(Q_0u)_n-Q_b u_n , Q_b v_{0,n}-v_b\big\rangle_{\partial T}\bigg| \nonumber\\
& = \bigg|\sum_{T \in \mathcal{T}_h} h_T^{-1}\big\langle (Q_0u)_n- u_n , Q_b v_{0,n}-v_b\big\rangle_{\partial T}\bigg| \nonumber\\
& \leq \bigg(\sum_{T \in \mathcal{T}_h}\left(h_T^{-2}\left\|\bm Q_0 \bm u - \bm u\right\|_T^2+\left\|\nabla\left(\bm Q_0 \bm u-\bm u\right)\right\|_T^2\right)\bigg)^{1/2}|\!|\!|\bm v|\!|\!|\nonumber\\
& \leq C h  \|\bm u\|_{2}|\!|\!|\bm v|\!|\!|.\label{eq:s}
\end{align}
This completes the proof.
\end{proof}
Denote
$$
W_h^0=\left\{ q \in L^2(\Omega): \left.q\right|_T \in P_{0}(T)\text{ for all } T \in \mathcal{T}_h \text{ and }\int_\Omega q \, dx = 0 \right\}.
$$
To establish the error estimate for the stress tensor, we present the following lemma. 
\begin{lemma}\label{infsup}
There exists a positive constant $\beta$ independent of $h$ such that
\begin{equation}\label{eq:inf}
\sup _{\bm{v} \in \bm V_h^0} \frac{(\nabla_w \cdot \bm{v}, \rho)}{|\!|\!| \bm{v}|\!|\!| } \geq \beta\|\rho\|,\quad \forall \rho \in W_h^0.
\end{equation}
\end{lemma}
\begin{proof}For any  $\rho \in W_h^0 \subset L_0^2(\Omega)$, there exists a vector-valued function $\tilde{\bm v} \in\left[H_0^1(\Omega)\right]^d$ such that
\begin{equation}\label{eqn-5}
\frac{(\nabla \cdot \tilde{\bm v}, \rho)}{\| \tilde{\bm v}\|_1} \geq C\|\rho\|,
\end{equation}
    where $C>0$ is a constant depending only on the domain $\Omega$. 
  By setting $\bm v=\{\bm v_0,v_b\}=\bm Q_h \tilde{\bm v} \in \bm V_h^0$,  it follows from 
    \eqref{eq:Qhg}, the trace inequality, and \eqref{eqn-7.1} that
 $$  \begin{aligned}
\left\|\bm\epsilon_w( \bm v)\right\|_T^2 =& \big(\bm\epsilon_w\left(\bm Q_h \tilde{\bm v}\right), \bm\epsilon_w(\bm v)\big)_T\\
 =&\big( \mathbb Q_h \bm\epsilon( \tilde{\bm v}),\bm\epsilon_w(\bm v)\big)_T+\big\langle \boldsymbol{n} \times (\bm Q_0 \tilde{\bm v}  - \tilde{\bm v}) ,\boldsymbol{n} \times \bm\epsilon_w({\bm v}) \cdot \bm{n}\big\rangle_{\partial T}\\
 \leq&\ \|\mathbb Q_h \bm\epsilon( \tilde{\bm v})\|_T \|\bm\epsilon_w({\bm v})\|_T  +\big(h_T^{-1}\|\boldsymbol{n} \times (\bm Q_0 \tilde{\bm v}  - \tilde{\bm v}) \|_{\partial T}^2 \big)^{1/2} \big(h_T\|\boldsymbol{n} \times \bm\epsilon_w({\bm v}) \cdot \bm{n}\|_{\partial T}^2 \big)^{1/2}\\
 \leq &\ C \|\tilde{\bm v}\|_{1,T}\|\bm\epsilon_w({\bm v})\|_T,
\end{aligned}
$$
which leads to \begin{equation}\label{eqn-7}
\left\|\bm\epsilon_w({\bm v})\right\|_T \leq C\| \tilde{\bm v}\|_{1,T}.
\end{equation}
Note that $\bm v_0 = \bm Q_0\tilde{\bm v}$ and $v_b = Q_b\tilde v_n$. 
Using the trace inequality, the definition of $Q_b$, and \eqref{eqn-7.1} yields
\begin{equation}
\begin{aligned}\label{eqn-8}
 &\sum_{T \in \mathcal{T}_h} h_T^{-1}\left\|Q_bv_{0,n}-v_b\right\|_{\partial T}^2 =\sum_{T \in \mathcal{T}_h} h_T^{-1}\left\|Q_bv_{0,n}-Q_b \tilde v_n)\right\|_{\partial T}^2 \\
 \leq& \ \sum_{T \in \mathcal{T}_h} h_T^{-1}\left\|v_{0,n}- \tilde v_n\right\|_{\partial T}^2 
 \leq C\sum_{T \in \mathcal{T}_h} h_T^{-1}\left\|\bm Q_0\tilde{\bm v}- \tilde{\bm v}\right\|_{\partial T}^2  {\leq C\|  \tilde{\bm v}\|_1^2.}
\end{aligned}
\end{equation}
By  collecting  \eqref{eqn-7} and \eqref{eqn-8}, we have 
{\begin{equation}\label{eqn-6}
|\!|\!|\bm v|\!|\!| \leq C\| \tilde{\bm v}\|_1.
\end{equation}}
It follows from \eqref{eq:Qhd} and the definition of $\mathcal Q_h$ that
$$
b(\bm v, \rho)=\left(\nabla_w \cdot\left(\bm Q_h \tilde{\bm v}\right), \rho\right)=\left(\mathcal Q_h(\nabla \cdot \tilde{\bm v}), \rho\right)=(\nabla \cdot \tilde{\bm v}, \rho),
$$
which, together with
 \eqref{eqn-5} and \eqref{eqn-6}, gives
$$
\frac{|b(\bm v, \rho)|}{|\!|\!| \bm v|\!|\!| } \geq \frac{|(\nabla \cdot \tilde{\bm v}, \rho)|}{C\|\tilde{\bm v}\|_1} \geq \beta\|\rho\|,
$$
where $\beta$ is a positive constant. 
\end{proof}

\begin{theorem}\label{th1}
The following error estimates hold true,
\begin{align}
|\!|\!|\bm e_h |\!|\!| & \leq C h (\|\bm u\|_2+\lambda\|\nabla\cdot\bm u\|_1) ,\label{eherr}\\
\| \bm \sigma_w (\bm e_h)\| & \leq C h (\|\bm u\|_2+\lambda\|\nabla\cdot\bm u\|_1),\label{diverr}
\end{align}
where $C$ is independent of the Lam\'e constant $\lambda$.
\end{theorem}
\begin{proof}
By letting $\bm{v}=\bm{e}_h$ in error equation \eqref{erreq} and using  Lemma \ref{le:erres}, we obtain
$$
C_1|\!|\!|\bm e_h |\!|\!|^2 \leq \bm a(\bm e_h, \bm e_h) \leq C_2 h (\|\bm u\|_2+\lambda\|\nabla\cdot\bm u\|_1) |\!|\!|\bm e_h |\!|\!|,
$$
which implies that
$$
|\!|\!|\bm e_h |\!|\!|  \leq C h (\|\bm u\|_2+\lambda\|\nabla\cdot\bm u\|_1).
$$
Note that $\nabla_w \cdot \bm e_h \in W_h^0$.
Then it follows from  Lemma \ref{infsup}, \eqref{erreq},  Lemma 
 \ref{le:erres}, and  \eqref{eherr} that
\begin{align*}
   & \lambda \|\nabla_w \cdot \bm e_h\|  \leq \sup _{\bm v \in \bm V_h^0} \frac{\lambda (\nabla_w \cdot \bm v,\nabla_w \cdot \bm e_h)}{\beta |\!|\!|\bm v |\!|\!|} \\
    = & \  \sup _{\bm v \in \bm V_h^0}\frac{ 2 \mu\sum_{i=1}^2\ell_i (\bm u, \bm v)+\lambda \ell_3 (\bm u, \bm v) -2 \mu\sum_{T \in \mathcal{T}_h}\left(\bm \epsilon_w(\bm e_h), \bm \epsilon_w(\bm v)\right)_T + s\left(\bm Q_h \boldsymbol{u}, \boldsymbol{v}\right)-s(\bm e_h,\bm v)}{\beta |\!|\!|\bm v |\!|\!|}\\
     \leq & \  C h (\|\bm u\|_2+\lambda\|\nabla\cdot\bm u\|_1),
\end{align*}
which, combined with \eqref{eherr} and \eqref{weakeps}, yields \eqref{diverr}.
\end{proof}
\begin{corollary}\label{re3} Under the assumption \eqref{regularity}, the following error estimates hold
    \begin{align*}
    |\bm u-\bm u_0|_1&\leq Ch\|\bm f\|,\\
\|\bm{\sigma} (\bm u)-\bm{\sigma}_w (\bm u_h) \| &\leq C h \|\bm f\|.
\end{align*}
\end{corollary}
\begin{proof}
 To estimate $|\bm u-\bm u_0|_1$, we
    apply \eqref{eherr}, \eqref{eqn-7.1}, Lemma \ref{lem:2},  \eqref{generalizedkorn}, and \eqref{regularity}. 
To estimate $\|\bm{\sigma} (\bm u)-\bm{\sigma}_w (\bm u_h) \|$, we begin with the triangle inequality, which gives  
\begin{equation}\label{ree3}
\begin{aligned}
\|\bm{\sigma} (\bm u)-\bm{\sigma}_w (\bm u_h)\|  
& \leq  \|\bm{\sigma} (\bm u)-\bm{\sigma}_w (\bm Q_h \bm u) \|+ \| \bm{\sigma}_w (\bm Q_h \bm u)-\bm{\sigma}_w (\bm u_h)\|\\
& = 2\mu \|\bm{\epsilon} (\bm u)-\bm{\epsilon}_w (\bm Q_h \bm u) \|+\lambda \|\nabla\cdot\bm u-\nabla_w\cdot(\bm Q_h\bm u)\|  + \| \bm{\sigma}_w (\bm e_h)\|.
\end{aligned}
\end{equation}
For the term \( \|\nabla\cdot\bm u-\nabla_w\cdot(\bm Q_h\bm u)\| \), applying \eqref{eq:Qhd} and \eqref{eqn-7.3} leads to  
\begin{align}\label{divest}
    \|\nabla\cdot\bm u-\nabla_w\cdot(\bm Q_h\bm u)\| = \|\nabla\cdot\bm u-\mathcal Q_h(\nabla\cdot\bm u)\|\leq Ch\|\nabla\cdot\bm u\|_1.
\end{align}
For the term \( \|\bm{\epsilon} (\bm u)-\bm{\epsilon}_w (\bm Q_h \bm u) \| \), we first apply \eqref{eq:Qhg}, the trace inequality, and the inverse inequality to obtain  
\[
\|\mathbb Q_h\bm\epsilon (\bm u)-\bm\epsilon_w (\bm Q_h \bm u)\| \leq C h \|\bm u\|_2,
\]
which, combined with \eqref{eqn-7.2}, yields
\begin{align}\label{epsest}
    \|\bm{\epsilon} (\bm u)-\bm{\epsilon}_w (\bm Q_h\bm u)\| \leq C h \|\bm u\|_2.
\end{align}
Substituting \eqref{divest} and \eqref{epsest} into \eqref{ree3} and incorporating \eqref{diverr}, we obtain  
\begin{align*}
\|\bm{\sigma} (\bm u)-\bm{\sigma}_w (\bm u_h) \| \leq C h \|\bm u\|_2 + C h \lambda \|\nabla \cdot \bm u\|_1,
\end{align*}
which, together with \eqref{regularity}, completes the proof.
\end{proof}

\begin{remark}
\cref{re3} demonstrates that Algorithm \ref{eg} is locking-free.
\end{remark}

Next, we derive an error estimate in the sense of $L^2$-norm through a duality argument. To this end, consider the dual problem of seeking $\bm{\phi}$ such that

\begin{subequations}
\begin{align}\label{e0}
-\nabla \cdot \bm \sigma(\bm{\phi}) & =\bm e_0 & & \text { in } \Omega, \\
\bm{\phi} & =0 & & \text { on } \partial\Omega, \label{e0b}
\end{align}
\end{subequations}
and assume that
\begin{align}\label{assump-reg-phi}
    \|\bm{\phi}\|_2 + \lambda \|\nabla \cdot \bm \phi\|_1\leq C\left\|\bm e_0\right\|.
\end{align}
\begin{theorem}\label{l2err}
Under the assumptions \eqref{assump-reg-phi} and \eqref{regularity}, the following estimate holds
\[\|\bm u-\bm u_0\|\leq Ch^2\|\bm f\|.\]
\end{theorem}

\begin{proof}
We multiply \eqref{e0} by $\bm e_0$ and integrate over $\Omega$ to get
\begin{align}\label{e02}
 \|\bm e_0\|^2=\left(-\nabla \cdot \bm \sigma(\bm{\phi}), \bm e_0\right) = 2 \mu \sum_{T \in \mathcal{T}_h}\big(\bm \epsilon ( \bm \phi), \bm\epsilon( \bm e_0)\big)_T +\lambda\sum_{T \in \mathcal{T}_h} (\nabla \cdot \bm \phi , \nabla \cdot \bm e_0)_T.   
\end{align}
For $(\bm \epsilon (\bm \phi),\bm\epsilon( \bm e_0))_T$, we apply the definition of $\mathbb Q_h$, \eqref{diff-eps}, and \eqref{eq:Qhg}  to derive
\begin{align}\label{epsphi}
(\bm \epsilon (\bm \phi),\bm\epsilon( \bm e_0))_T = &\ (\mathbb Q_h \bm \epsilon (\bm \phi),\bm\epsilon( \bm e_0))_T
\nonumber\\
=&\ (\mathbb Q_h\bm \epsilon (\bm \phi),\bm\epsilon_w(\bm e_h))_T+\langle\bm n\cdot\mathbb Q_h\bm \epsilon (\bm \phi)\cdot\bm n, (Q_be_{0,n}-e_b) \bm n_e\cdot\bm n\rangle_{\partial T}\nonumber\\
=&\ (\mathbb Q_h\bm \epsilon (\bm \phi),\mathbb Q_h\bm \epsilon (\bm u))_T-(\bm \epsilon_w(\bm Q_h \bm \phi),\bm \epsilon_w(\bm u_h))_T+\langle\bm n\cdot\mathbb Q_h\bm \epsilon (\bm \phi)\cdot\bm n, (Q_be_{0,n}-e_b) \bm n_e\cdot\bm n\rangle_{\partial T}\nonumber\\
&+\left\langle \bm{n} \times (\bm Q_0 \bm u  - \bm u) , \bm{n} \times \mathbb Q_h\bm \epsilon (\bm \phi)  \cdot \bm{n}\right\rangle_{\partial T}+\left\langle \bm{n} \times (\bm Q_0 \bm \phi  - \bm \phi) , \bm{n} \times \bm \epsilon_w (\bm u_h) \cdot \bm{n}\right\rangle_{\partial T}\nonumber\\
=&\ (\bm \epsilon (\bm \phi)-\mathbb Q_h\bm\epsilon (\bm \phi),\mathbb Q_h\bm\epsilon (\bm u)-\bm\epsilon (\bm u))_T+(\bm \epsilon (\bm \phi),\bm \epsilon (\bm u))_T-(\bm \epsilon_w(\bm u_h),\bm \epsilon_w(\bm Q_h \bm \phi))_T\nonumber\\
&+\langle\bm n\cdot\mathbb Q_h\bm \epsilon (\bm \phi)\cdot\bm n, (Q_be_{0,n}-e_b) \bm n_e\cdot\bm n\rangle_{\partial T}\nonumber\\
&+\left\langle \bm{n} \times (\bm Q_0 \bm u  - \bm u) , \bm{n} \times \mathbb Q_h\bm \epsilon (\bm \phi)  \cdot \bm{n}\right\rangle_{\partial T}+\left\langle \bm{n} \times (\bm Q_0 \bm \phi  - \bm \phi) , \bm{n} \times \bm \epsilon_w (\bm u_h) \cdot \bm{n}\right\rangle_{\partial T}
\end{align}
Similarly, for $(\nabla \cdot \bm \phi,\nabla \cdot \bm e_0)_T$, we apply the definition of $\mathcal Q_h$, \eqref{diff-div}, and \eqref{eq:Qhd}  to obtain
\begin{align}\label{divphi}
(\nabla \cdot \bm \phi,\nabla \cdot \bm e_0)_T = &\ (\mathcal Q_h \nabla \cdot \bm \phi,\nabla \cdot \bm e_0)_T
\nonumber\\
=&\ \langle\mathcal Q_h \nabla \cdot \bm \phi, (Q_be_{0,n}-e_b)\bm n_e\cdot\bm n\rangle_{\partial T}+(\mathcal Q_h \nabla \cdot \bm \phi,\nabla_w\cdot\bm e_h)_{T}\nonumber\\
=&\ \langle\mathcal Q_h \nabla \cdot \bm \phi, (Q_be_{0,n}-e_b)\bm n_e\cdot\bm n\rangle_{\partial T}+(\mathcal Q_h \nabla \cdot \bm \phi,\mathcal Q_h\nabla\cdot\bm u)_{T}-(\nabla_w\cdot \bm Q_h  \bm \phi,\nabla_w\cdot\bm u_h)_{T}\nonumber\\
=&\ \langle\mathcal Q_h \nabla \cdot \bm \phi, (Q_be_{0,n}-e_b)\bm n_e\cdot\bm n\rangle_{\partial T}+(\mathcal Q_h \nabla \cdot \bm \phi-\nabla \cdot \bm \phi,\nabla\cdot\bm u-\mathcal Q_h \nabla \cdot \bm u)_{T}\nonumber\\
&+ (\nabla \cdot \bm \phi,\nabla\cdot\bm u)_{T}-(\nabla_w\cdot \bm Q_h  \bm \phi,\nabla_w\cdot\bm u_h)_{T}
\end{align}
Plugging \eqref{epsphi} and \eqref{divphi} into \eqref{e02} leads to
\begin{equation}\label{e0l2}
\begin{aligned}
\|\bm e_0\|^2 =&\ 2 \mu \sum_{T \in \mathcal{T}_h}\big(\bm \epsilon ( \bm \phi), \bm \epsilon ( \bm e_0)\big)_T +\lambda\sum_{T \in \mathcal{T}_h} (\nabla \cdot \bm \phi , \nabla \cdot \bm e_0)_T\\
=& -2\mu (\bm\epsilon ( \bm \phi)-\mathbb Q_h \bm\epsilon ( \bm \phi),\bm\epsilon ( \bm u)-\mathbb Q_h \bm\epsilon ( \bm u))- 2\mu \ell_1(\bm\phi,\bm e_h)\\ 
&+ 2 \mu \sum_{T \in \mathcal{T}_h} \big \langle \bm n \times (\bm Q_0 \bm \phi-\bm \phi) ,  \bm{n} \times (\bm \epsilon_w (\bm u_h)-\bm \epsilon (\bm u)) \cdot \bm{n}\big\rangle_{\partial T}\\
&+ 2 \mu \sum_{T \in \mathcal{T}_h} \big \langle \bm n \times (\bm Q_0 \bm u-\bm u) ,  \bm{n} \times (\mathbb Q_h\bm\epsilon(\bm\phi)-\bm\epsilon(\bm\phi)) \cdot \bm{n}\big\rangle_{\partial T}\\
&-\lambda  (\nabla \cdot \bm \phi-\mathcal Q_h \nabla \cdot \bm \phi,\nabla \cdot \bm u-\mathcal Q_h \nabla \cdot \bm u)
- \lambda \ell_3(\bm\phi,\bm e_h)\\
&+ (\bm f,\bm \phi)-(\bm f, \bm Q_0 \bm \phi)- s\left( \bm e_h, \bm Q_h \bm \phi \right)+s\left( \bm Q_h \bm u, \bm Q_h \bm \phi \right)
\end{aligned}
\end{equation}
We estimate each term on the right-hand side of \eqref{e0l2}. For the first and fifth terms, we use \eqref{eqn-7.2} and \eqref{eqn-7.3} to obtain
\begin{align}\label{r1}
   & \Big|2\mu \big(\bm\epsilon ( \bm u)-\mathbb Q_h \bm\epsilon ( \bm u),\bm\epsilon ( \bm \phi)-\mathbb Q_h \bm\epsilon ( \bm \phi)\big)+\lambda \big(\nabla \cdot \bm \phi-\mathcal Q_h \nabla \cdot \bm \phi,\nabla \cdot \bm u-\mathcal Q_h \nabla \cdot \bm u\big)\Big|\nonumber\\
    \leq &\ Ch^2\|\bm u\|_2\|\bm\phi\|_2+C \lambda h^2\|\bm u\|_2\|\nabla \cdot \bm \phi \|_1.
\end{align}
For the third term, we use Cauchy-Schwarz inequality,  trace inequality, \eqref{eqn-7.1}, and \eqref{eqn-7.2} to have
\begin{align}\label{r2}
& \Big| 2 \mu \sum_{T \in \mathcal{T}_h} \big \langle \bm n \times (\bm Q_0 \bm u-\bm u) ,  \bm{n} \times (\mathbb Q_h\bm\epsilon(\bm\phi)-\bm\epsilon(\bm\phi)) \cdot \bm{n}\big\rangle_{\partial T}\Big| \nonumber\\
 \leq & \ C\bigg(\sum_{T \in \mathcal{T}_h} h_T^{-1}\| \bm u -\bm Q_0\bm u\|_{\partial T}^2\bigg)^{1/2} \bigg(\sum_{T \in \mathcal{T}_h} h_T\| \bm \epsilon (\bm \phi) - \mathbb{Q}_h\bm \epsilon (\bm \phi) \|_{\partial T}^2\bigg)^{1/2}\nonumber\\
 \leq  &\ C h^2 \|\bm u\|_2\|\bm \phi\|_2. 
\end{align}
For the fourth term, we use Cauchy-Schwarz inequality, trace inequality, \eqref{eqn-7.1}, \eqref{eherr}, and \eqref{epsest} to have
\begin{align}\label{r3}
& \Big| 2 \mu \sum_{T \in \mathcal{T}_h} \big \langle \bm n \times (\bm Q_0 \bm \phi-\bm \phi) ,  \bm{n} \times (\bm \epsilon_w (\bm u_h)-\bm \epsilon (\bm u)) \cdot \bm{n}\big\rangle_{\partial T} \Big|  \nonumber\\ 
 \leq & \  \bigg(\sum_{T \in \mathcal{T}_h} h_T^{-1} \left\|\bm n \times(\bm Q_0 \bm \phi - \bm \phi )\right\|_{\partial T}^2\bigg)^{1/2} \bigg(\sum_{T \in \mathcal{T}_h} h_T\left\|\bm \epsilon_w (\bm u_h)-\bm \epsilon (\bm u)\right\|_{\partial T}^2\bigg)^{1 / 2} \nonumber\\
 \leq & \ Ch^2  \|\bm \phi \|_2 \|\bm u\|_2
\end{align}
For the term $ (\bm f,\bm \phi)-(\bm f, \bm Q_0 \bm \phi)$,
\begin{align}\label{r4}
\big|(\bm f,\bm \phi)-(\bm f, \bm Q_0 \bm \phi)\big| \leq  C   \|\bm  f \|\| \bm \phi -\bm Q_0 \bm \phi\| \leq Ch^2 \|\bm f\| \|\bm \phi\|_2.
\end{align}
For the last term, we apply the Cauchy-Schwarz inequality, trace inequality, and \eqref{eqn-7.1} to derive
\begin{align}\label{r5}
\big|s\left(\bm Q_h \bm u, \bm Q_h \bm \phi\right)\big| & = \bigg|\sum_{T \in \mathcal{T}_h} h_T^{-1}\big\langle Q_b(Q_0u)_n-Q_b u_n , Q_b(Q_0\phi)_n-Q_b \phi_n\big\rangle_{\partial T}\bigg| \nonumber\\
& \leq \bigg(\sum_{T \in \mathcal{T}_h} h_T^{-1}\|(Q_0u)_n-u_n\|_{\partial T}^2\bigg)^{1/2}\bigg(\sum_{T \in \mathcal{T}_h} h_T^{-1}\|(Q_0\phi)_n- \phi_n\|_{\partial T}^2\bigg)^{1/2}\nonumber\\
& \leq C h^2  \|\bm u\|_{2}\|\bm \phi\|_{2}.
\end{align}
We complete the proof by combining \eqref{r1} -- \eqref{r5}, applying Lemma \ref{le:erres} to the term $-2\mu\ell_1(\bm\phi,\bm e_h)-\lambda\ell_3(\bm\phi,\bm e_h)-s(\bm e_h,\bm Q_h\bm\phi)$, and utilizing \eqref{eherr}, \eqref{assump-reg-phi}, and \eqref{regularity}.
\end{proof}

\section{Numerical Examples}\label{sec:num}
 In this section, we present numerical examples to verify the convergence and the locking-free property of the proposed method.

\begin{example}[2D accuracy test]\label{Vertex}  \rm
In this example, we validate the optimal convergence rates of our proposed method using a smooth displacement solution defined on $\Omega = (0,1)^2$:
$$
\boldsymbol{u}=\left(\begin{array}{c} \sin x \sin y+\frac{x}{\lambda}\\ \cos x \cos y+\frac{y}{\lambda} \end{array}\right).
$$
We impose the pure Dirichlet boundary condition $\bm{u}=\bm{u}_D$. The body force $\bm f$ and $\bm{u}_D$ are computed from the exact solution.

As $\lambda \rightarrow \infty$, a straightforward calculation shows that  $\nabla \cdot \boldsymbol{u} \rightarrow 0$. Consequently, this solution exhibits volumetric locking for a large $\lambda$. In order to show the optimal convergence rate along with   the locking-free property of our method, we consider the elasticity problem with two different material parameters on a uniform triangular mesh: $\lambda=1$ for compressible elasticity and $\lambda=10^6$ for nearly incompressible elasticity with $\mu=1$ in both cases. 
The numerical results in Table \ref{vortex tabel} validate the theoretical findings in Corollary \ref{re3} and Theorem \ref{l2err}, confirming that the proposed method achieves optimal convergence rates in both compressible and nearly incompressible cases, and demonstrating its locking-free property. Moreover, unlike the EG method in \cite{su2024parameter}, the proposed EG method provides an accurate approximation of stress without requiring any post-processing techniques.


\begin{table}[h]
\centering
\caption{Example \ref{Vertex}: 
Numerical results of the proposed method with $\lambda=1$ and $\lambda=10^{6}$.}
\begin{tabular}{|c|c|c|c|c|c|c|}
\hline $h$ & $\left\|\bm{u}-\bm{u}_0\right\|$ &Rate &$|\bm{u}-\bm{u}_0|_1 $ &Rate&$\|\bm \sigma (\bm{u})-\bm \sigma_w (\bm{u}_h)\|$  &Rate \\
\hline \multicolumn{7}{|c|}{$\lambda=1,\mu=1$} \\
\hline 1/8 &$1.665 \mathrm{e}-03$ &        & $7.032  \mathrm{e}-02$ &        & $1.152 \mathrm{e}-01$ &       \\
\hline 1/16&$3.882 \mathrm{e}-04$ & 2.1010 & $3.540  \mathrm{e}-02$ & 0.9903 & $5.879 \mathrm{e}-02$ & 0.9713\\
\hline 1/32&$9.338 \mathrm{e}-05$ & 2.0557 & $1.776  \mathrm{e}-02$ & 0.9946 & $2.967 \mathrm{e}-02$ & 0.9863\\
\hline 1/64&$2.287 \mathrm{e}-05$ & 2.0294 & $8.901  \mathrm{e}-03$ & 0.9971 & $1.491 \mathrm{e}-02$ & 0.9933\\
\hline \multicolumn{7}{|c|}{$\lambda=10^6,\mu=1$} \\ 
\hline 1/8 &$1.701 \mathrm{e}-03$ &         & $ 7.021 \mathrm{e}-02$ &         & $1.219 \mathrm{e}-01$ &       \\
\hline 1/16&$3.935 \mathrm{e}-04$ &  2.1124 & $ 3.528 \mathrm{e}-02$ &  0.9929 & $6.145 \mathrm{e}-02$ & 0.9886\\
\hline 1/32&$9.423 \mathrm{e}-05$ &  2.0620 & $ 1.769 \mathrm{e}-02$ &  0.9959 & $3.088 \mathrm{e}-02$ & 0.9927\\
\hline 1/64&$2.302 \mathrm{e}-05$ &  2.0329 & $ 8.859 \mathrm{e}-03$ &  0.9976 & $1.549 \mathrm{e}-02$ & 0.9954\\
\hline
\end{tabular}
\label{vortex tabel}
\end{table}

\end{example}
\begin{example}[3D accuracy test] \label{3D}  \rm
In this example, we verify the optimal convergence rates of our proposed method using a smooth displacement solution defined on $\Omega=(0,1)^3$ :
$$
\boldsymbol{u}=\left(\begin{array}{c}
2 \sin x \sin y \sin z+\frac{x}{\lambda} \\
\cos x \cos y \sin z+\frac{y}{\lambda} \\
\cos x \sin y \cos z+\frac{z}{\lambda}
\end{array}\right)
$$ 
We impose the pure Dirichlet boundary condition $\bm{u}=\bm{u}_D$. The body force $\bm f$ and $\bm{u}_D$ are computed from the exact solution.

Similar to Example \ref{Vertex}, we consider both the compressible case $(\lambda=1, \mu=1)$ and nearly incompressible case $\left(\lambda=10^6, \mu=1\right)$ to evaluate the convergence and locking-free property of our method. The numerical results on a uniform mesh are presented in Table \ref{3D tabel}, verifying the theoretical findings in  Corollary \ref{re3} and Theorem \ref{l2err}.
\begin{table}[h]
\centering
\caption{Example \ref{3D}: 
Numerical results of the proposed method with $\lambda=1$ and $\lambda=10^{6}$.}
\begin{tabular}{|c|c|c|c|c|c|c|}
\hline $h$ & $\left\|\bm{u}-\bm{u}_0\right\|$ &Rate &$|\bm{u}-\bm{u}_0|_1 $ &Rate&$\|\bm \sigma (\bm{u})-\bm \sigma_w (\bm{u}_h)\|$  &Rate \\
\hline \multicolumn{7}{|c|}{$\lambda=1,\mu=1$} \\
\hline 1/4  &$1.079 \mathrm{e}-02$ &         & $2.044  \mathrm{e}-01$ &         & $3.179 \mathrm{e}-01$ &       \\
\hline 1/8  &$2.481 \mathrm{e}-03$ &  2.1217 & $1.009  \mathrm{e}-01$ &  1.0176 & $1.529 \mathrm{e}-01$ &  1.0554\\
\hline 1/12 &$1.069 \mathrm{e}-03$ &  2.0752 & $6.701  \mathrm{e}-02$ &  1.0111 & $1.008 \mathrm{e}-01$ &  1.0275\\
\hline 1/16 &$5.926 \mathrm{e}-04$ &  2.0524 & $5.015  \mathrm{e}-02$ &  1.0076 & $7.526 \mathrm{e}-02$ &  1.0175\\
\hline \multicolumn{7}{|c|}{$\lambda=10^6,\mu=1$} \\ 
\hline 1/4  &$1.049 \mathrm{e}-02$ &         & $ 2.153 \mathrm{e}-01$ &         & $5.563 \mathrm{e}-01$ &       \\
\hline 1/8  &$2.408 \mathrm{e}-03$ &  2.1233 & $ 1.031 \mathrm{e}-01$ &  1.0623 & $2.018 \mathrm{e}-01$ & 1.4630\\
\hline 1/12 &$1.035 \mathrm{e}-03$ &  2.0816 & $ 6.783 \mathrm{e}-02$ &  1.0324 & $1.198 \mathrm{e}-01$ & 1.2859\\
\hline 1/16 &$5.728 \mathrm{e}-04$ &  2.0580 & $ 5.058 \mathrm{e}-02$ &  1.0201 & $8.531 \mathrm{e}-02$ & 1.1807\\
\hline
\end{tabular}
\label{3D tabel}
\end{table}


\end{example}
\begin{example}[A solution with a corner singularity and a large $\lambda$ ]\label{Lshaped}  
\rm
 We consider the linear elasticity problem \eqref{eq:l1} on an $L$-shaped domain $\Omega=(-1,1)^2 /([0,1] \times[-1,0])$. The exact displacement $\boldsymbol{u}=\left(u_1, u_2\right)^{\mathrm{T}}$ is chosen  in polar coordinates $(r, \theta)$ as follows
$$
\begin{aligned}
& u_1=\frac{1}{2 \mu} r^\gamma\Big(\cos \theta\big(-(1+\gamma) \cos (\theta+\gamma\theta)+(k-\gamma) Q \cos (\theta-\gamma\theta)\big) \\
& -\sin \theta\big((1+\gamma) \sin (\theta+\gamma\theta)-(k+\gamma) Q \sin (\theta-\gamma\theta)\big)\Big), \\
& u_2=\frac{1}{2 \mu} r^\gamma\Big(\sin \theta\big(-(1+\gamma) \cos (\theta+\gamma\theta)+(k-\gamma) Q \cos (\theta-\gamma\theta)\big) \\
& +\cos \theta\big((1+\gamma) \sin (\theta+\gamma\theta)-(k+\gamma) Q \sin (\theta-\gamma\theta)\big)\Big),
\end{aligned}
$$
where
\begin{itemize}
 \item $k=3-4 \nu$ with the Poisson ratio $\nu=\frac{\lambda}{2(\lambda+\mu)}$,
 \item $\gamma$ is obtained by solving equation $$ \sin \left(\frac{3 \pi}{2}\gamma \right)+\gamma \sin \left(\frac{3 \pi}{2}\right)=0, $$
 \item $Q$ is given by $Q=-\frac{\cos \left((\gamma+1) \frac{3 \pi}{4}\right)}{\cos \left((\gamma-1) \frac{3 \pi}{4}\right)}$.
\end{itemize}

As in \cite{yi2022locking}, we choose $\gamma=0.5444837367$ and $Q=0.5430755688$. The body force $\boldsymbol{f}=(0,0)^{\mathrm{T}}$. We set $\mu=1$ and $\lambda=10^6$ and consider the pure Dirichlet boundary condition.

Based on the results in \cite{babuvska1987hp}, we have \( \bm u \in\left[H^{1+\gamma-\varepsilon}(\Omega)\right]^2 \) and \(\bm \sigma \in\left[H^{\gamma-\varepsilon}(\Omega)\right]^{2 \times 2} \) for any \( \varepsilon > 0 \). Consequently, the expected convergence rate of the proposed method in the \( H^1 \)-seminorm is approximately 0.54.  
To validate this, we conduct numerical experiments on a nonuniform triangular mesh, as depicted in Figure \ref{figL}, which depicts the level-0 mesh. The next-level mesh is obtained by connecting the midpoints of the three edges of each triangle. The numerical results presented in Table \ref{Lnum} confirm the expected convergence rate.


\begin{figure}[htbp]
    \centering
\includegraphics[width=0.5 \textwidth]{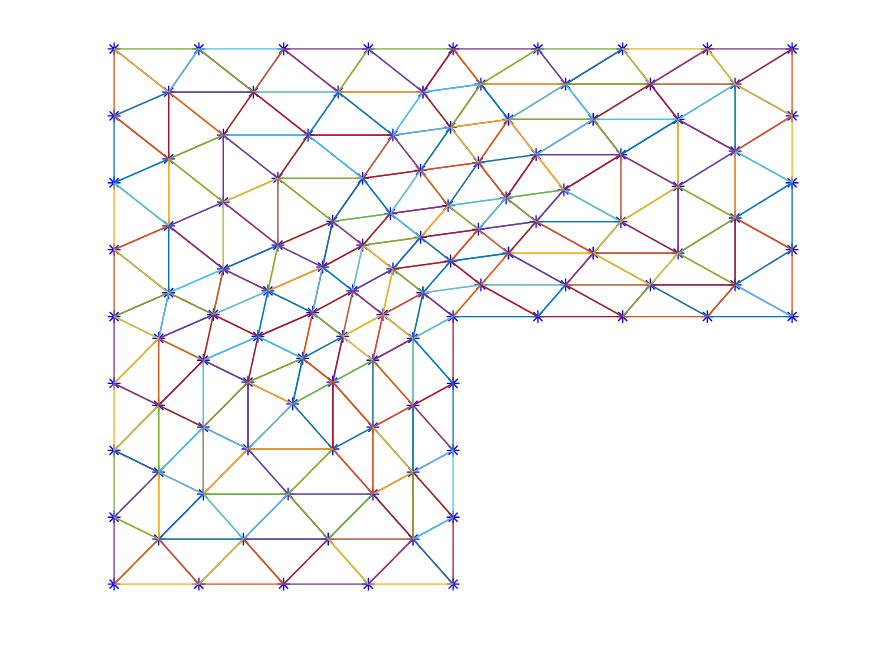}
    \caption{Example \ref{Lshaped}: The mesh of level 0 for the L-shaped domain.}
    \label{figL}
\end{figure}

\begin{table}[h]
\centering
\caption{Example \ref{Lshaped}: Numerical results of the proposed method with $\mu = 1$ and $\lambda = 10^6$.}
\begin{tabular}{|c|c|c|c|c|c|c|c|}
\hline Mesh level & Number of cells & $\left\|\bm{u}-\bm{u}_0\right\|$ &Rate &$|\bm{u}-\bm{u}_0|_1 $ &Rate &$\|\bm \sigma (\bm{u})-\bm \sigma_w (\bm{u}_h)\|$ &Rate \\
\hline 2 & 3072  &$ 3.879 \mathrm{e}-03$ &        & $2.486  \mathrm{e}-01$ &          &$ 1.881 \mathrm{e}+01$ &        \\
\hline 3 & 12288 &$ 1.373 \mathrm{e}-03$ & 1.4981 & $1.710  \mathrm{e}-01$ & 0.5397   &$ 6.641 \mathrm{e}+00$ & 1.5019 \\
\hline 4 & 49152 &$ 4.912 \mathrm{e}-04$ & 1.4831 & $1.174  \mathrm{e}-01$ & 0.5420   &$ 2.769 \mathrm{e}+00$ & 1.2617 \\
\hline 5 & 196608&$ 1.786 \mathrm{e}-04$ & 1.4597 & $8.062  \mathrm{e}-02$ & 0.5431   &$ 1.828 \mathrm{e}+00$ & 0.5994 \\
\hline
\end{tabular}
\label{Lnum}
\end{table}

\end{example}

\begin{example}[Cook's membrane test]\label{Cook}
\rm
The Cook’s membrane problem, as discussed in \cite{cook1974improved, piltner2000triangular}, is a classical benchmark for assessing element performance under combined bending and shear forces with moderate deformation. It is particularly important for evaluating volumetric locking effects. As illustrated in  Figure \ref{figC}, the problem consists of a tapered panel that is clamped on the left side while subjected to a shearing load on the right side. Specifically, 
we set
 \begin{itemize}
 \item $\bm {f}=(0,0)^{\mathrm{T}}$,
 \item Dirichlet boundary condition on the left boundary $\bm {u}=(0,0)^{\mathrm{T}}$,
 \item  vanishing Neumann boundary condition on the top and bottom boundary $\bm {\sigma} \bm {n}=(0,0)^{\mathrm{T}}$,
\item nonvanishing Neumann boundary condition on the right boundary $\bm {\sigma} \bm {n}=\left(0, \frac{1}{16}\right)^{\mathrm{T}}$.
 \end{itemize}
 
As in \cite{SEVILLA201943,su2024parameter}, we consider two cases to validate the performance of our EG method. The first case involves a compressible material with Young modulus $E=1$ and Poisson ratio $\nu=\frac{1}{3}$, and the second case involves a nearly incompressible material with Young modulus $E=1.12499998125$ and Poisson ratio $\nu=0.499999975$.


We present the numerical displacements for the two cases in Figure \ref{fig1} and Figure \ref{fig2}, which are in good agreement with those in \cite{SEVILLA201943,su2024parameter}. Additionally, in Figure \ref{fig3}, we plot the second component $u_2$ of the numerical displacement at $(48,52)$ against the number of DoFs. As the mesh is refined, 
$u_2(48,52)$ converges the reference value as documented in \cite{SEVILLA201943}. 

In plane strain,  the Von Mises stress is defined as
$$
\sigma_{\mathrm{VM}}=\sqrt{\frac{1}{2}\left(\sigma_{11}-\sigma_{22}\right)^2+\frac{1}{2}\left(\sigma_{33}-\sigma_{22}\right)^2+\frac{1}{2}\left(\sigma_{11}-\sigma_{33}\right)^2+3 \sigma_{12}^2},
$$
where $\sigma_{33}=\lambda \operatorname{div} \boldsymbol{u}$. Figure
\ref{figsig} depicts the Von Mises stress distributions for both compressible and nearly incompressible cases, revealing a singularity at the upper-left corner, which aligns with the findings reported in \cite{su2024parameter}.

\begin{figure}[htbp]
\centering
 \subfloat[Domain] {
\includegraphics[width=0.5\textwidth]{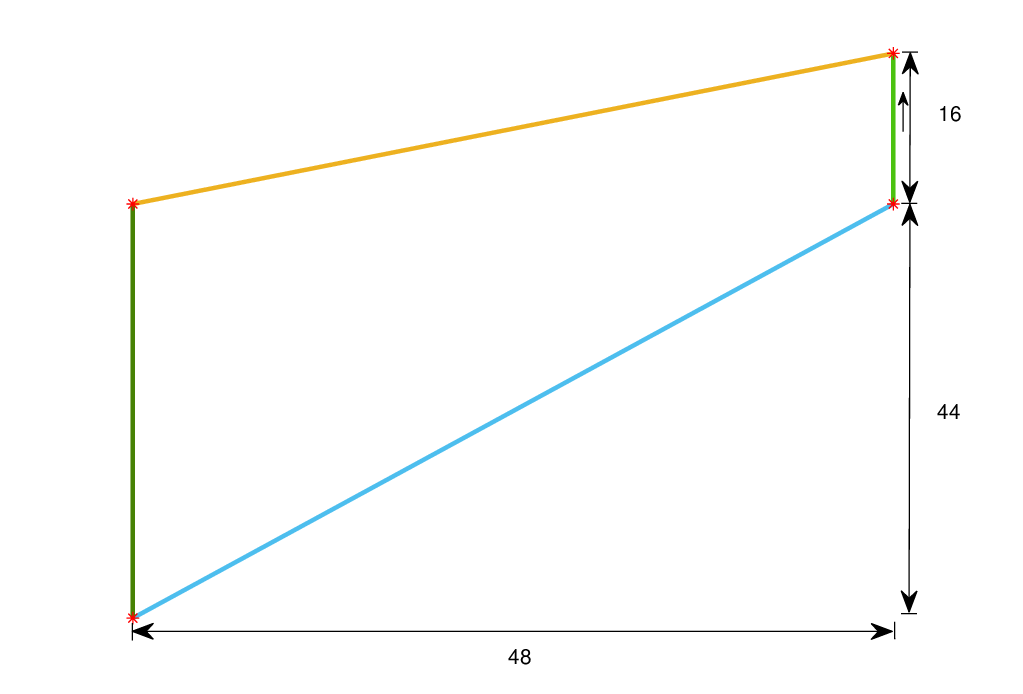}}
\subfloat[Mesh of level 1 of the domain] {
\centering
\includegraphics[width=0.45\textwidth]{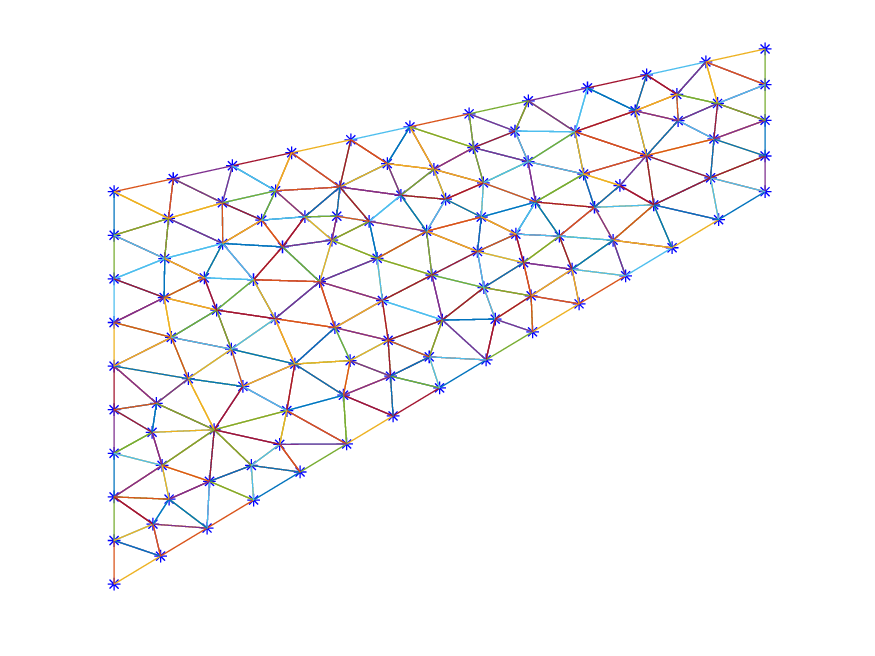}}
\caption{Example \ref{Cook}: The domain and mesh.
}
 \label{figC}
\end{figure}

\begin{figure}[htbp]
\centering
 \subfloat[$u_1$] {
\includegraphics[width=0.45\textwidth]{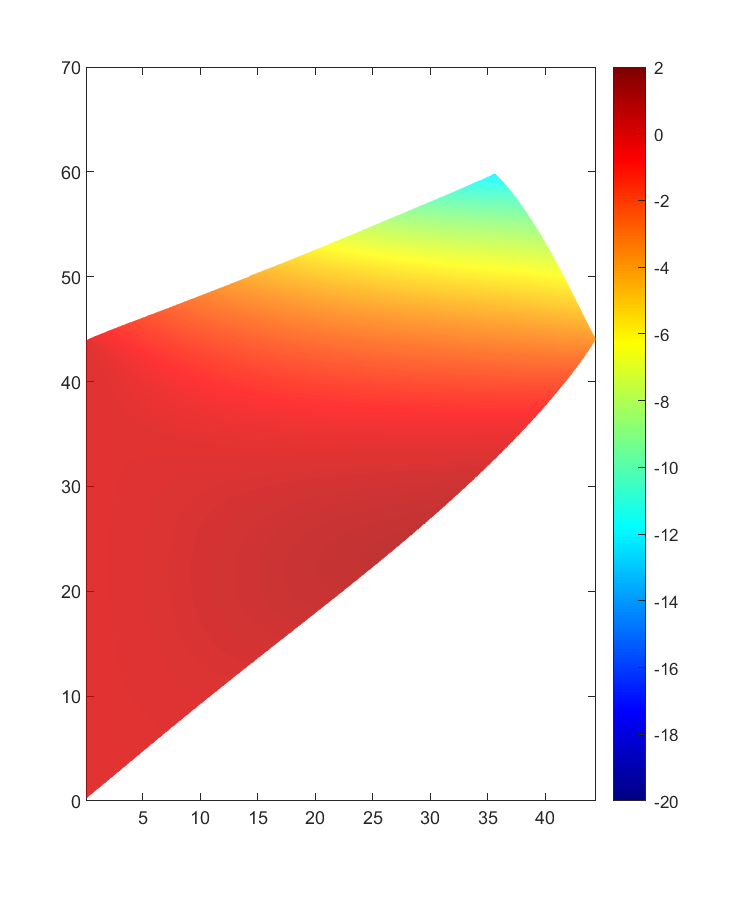}}
\subfloat[$u_2$] {
\centering
\includegraphics[width=0.45\textwidth]{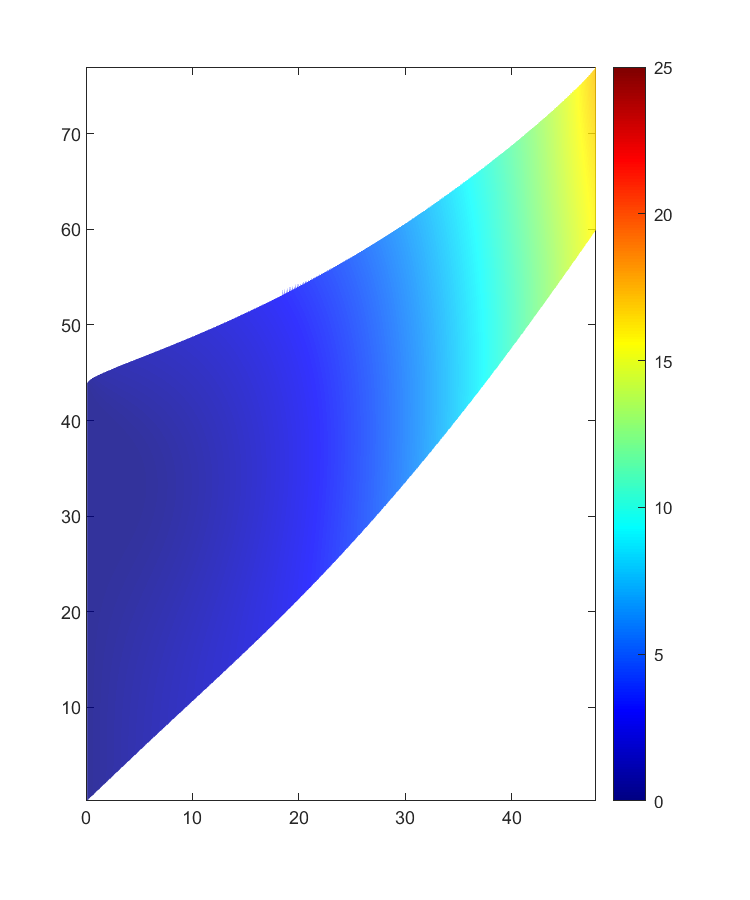}}
\caption{Example \ref{Cook}: The numerical displacement in nearly incompressible case on mesh of level 5.}
 \label{fig1}
\end{figure}
\begin{figure}[htbp]
\centering
 \subfloat[$u_1$] {
\includegraphics[width=0.45\textwidth]{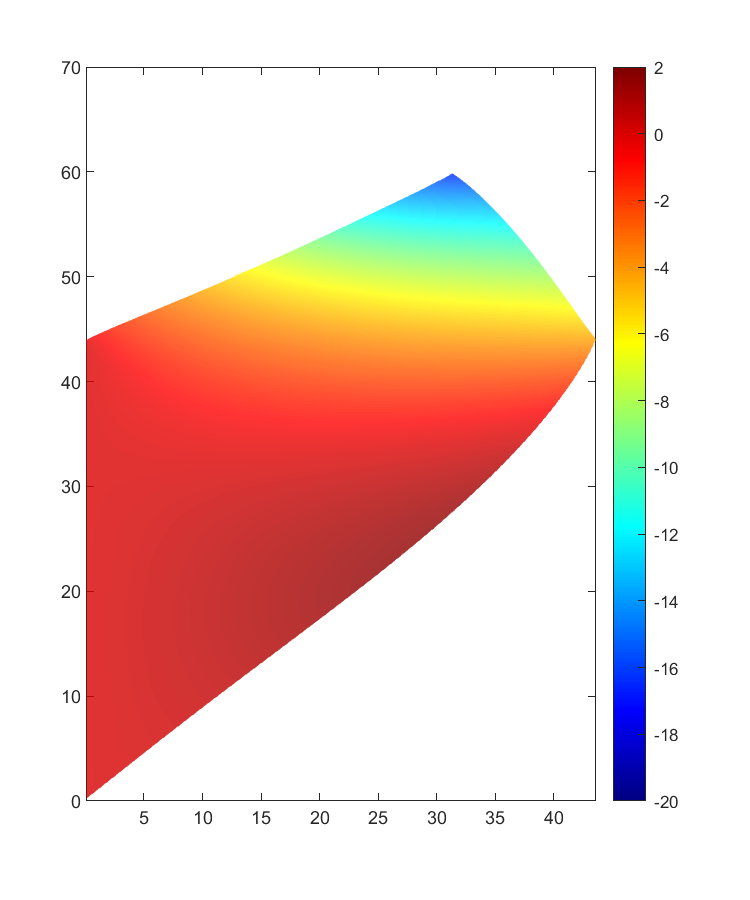}}
\subfloat[$u_2$]  {
\centering
\includegraphics[width=0.45\textwidth]{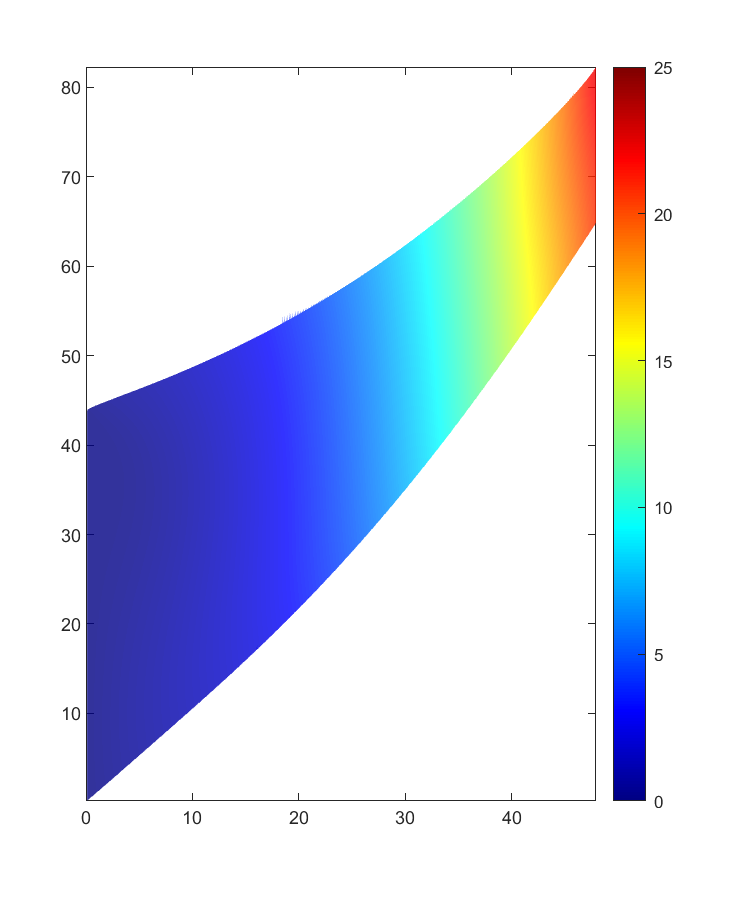}
 }\caption{Example \ref{Cook}: The numerical displacement in compressible case on mesh of level 5.}
 \label{fig2}
\end{figure}

\begin{figure}[htbp]
\centering
 \subfloat[$u_2$ in nearly incompressible case] {
\includegraphics[width=0.5\textwidth]{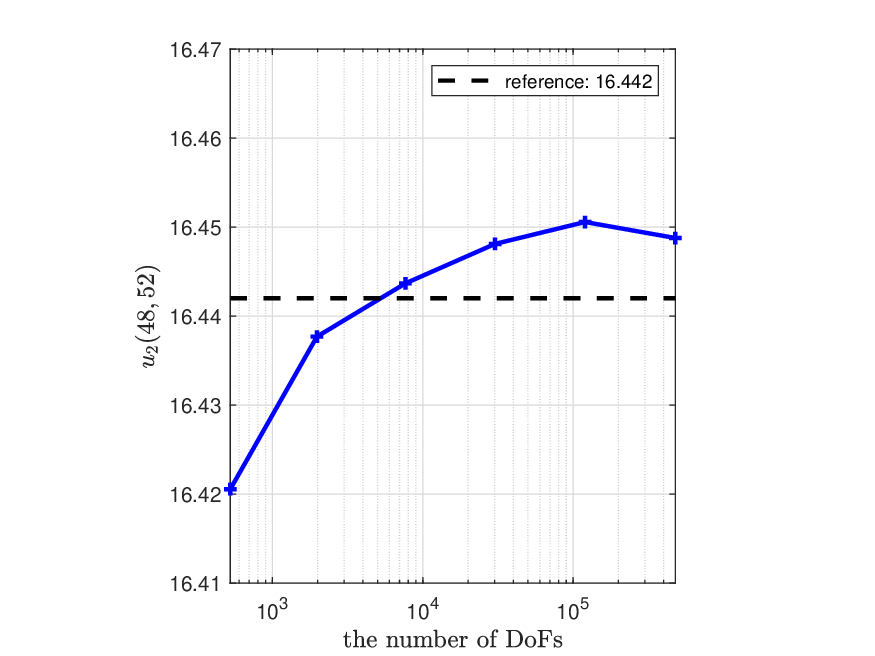}}
\subfloat[$u_2$ in compressible case]  {
\centering
\includegraphics[width=0.5\textwidth]{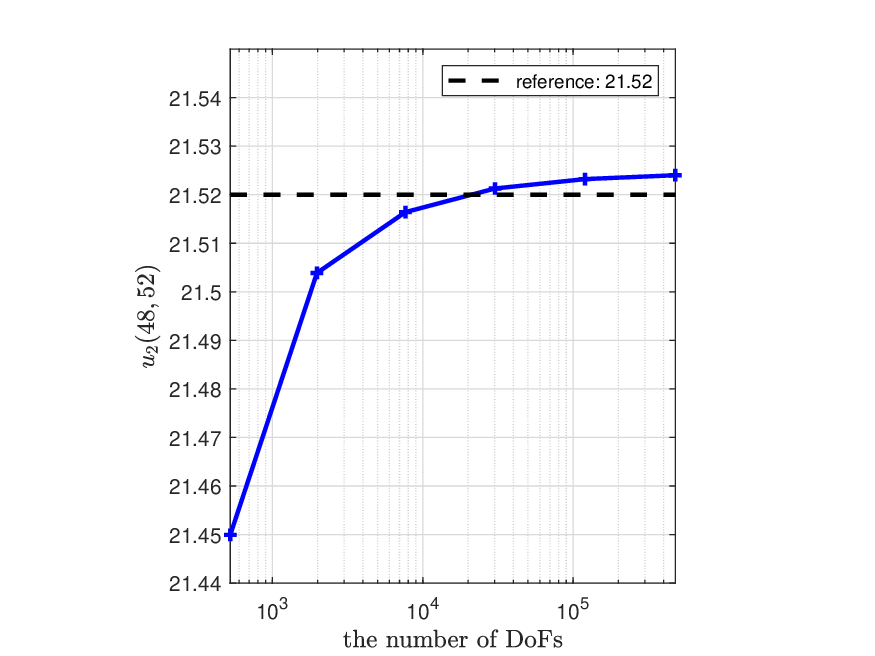}
 }\caption{Example \ref{Cook}: The numerical displacement at (48, 52).}
 \label{fig3}
\end{figure}

\begin{figure}[htbp]
\centering
 \subfloat[compressible case] {
\includegraphics[width=0.5\textwidth]{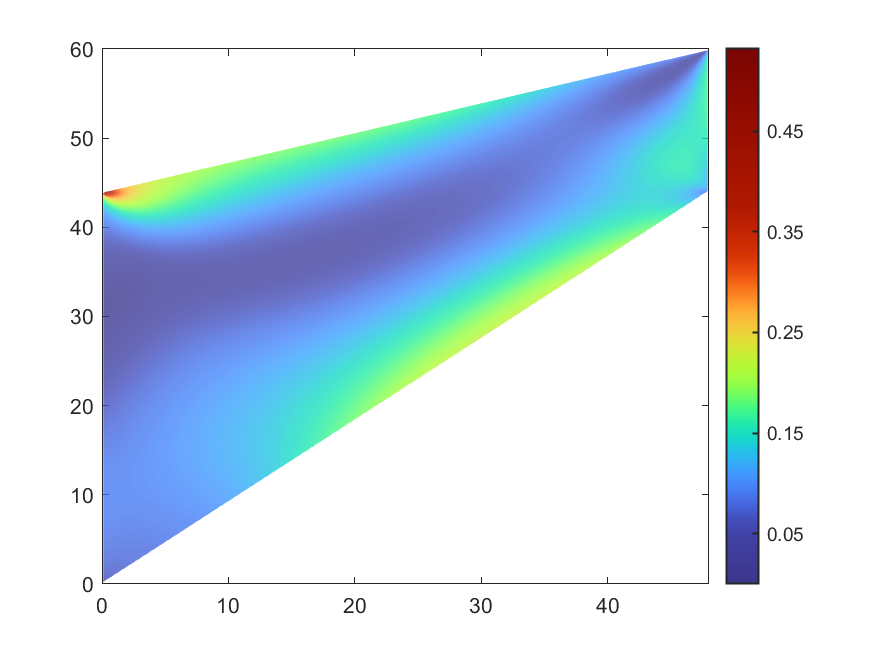}}
\subfloat[incompressible case] {
\centering
\includegraphics[width=0.5\textwidth]{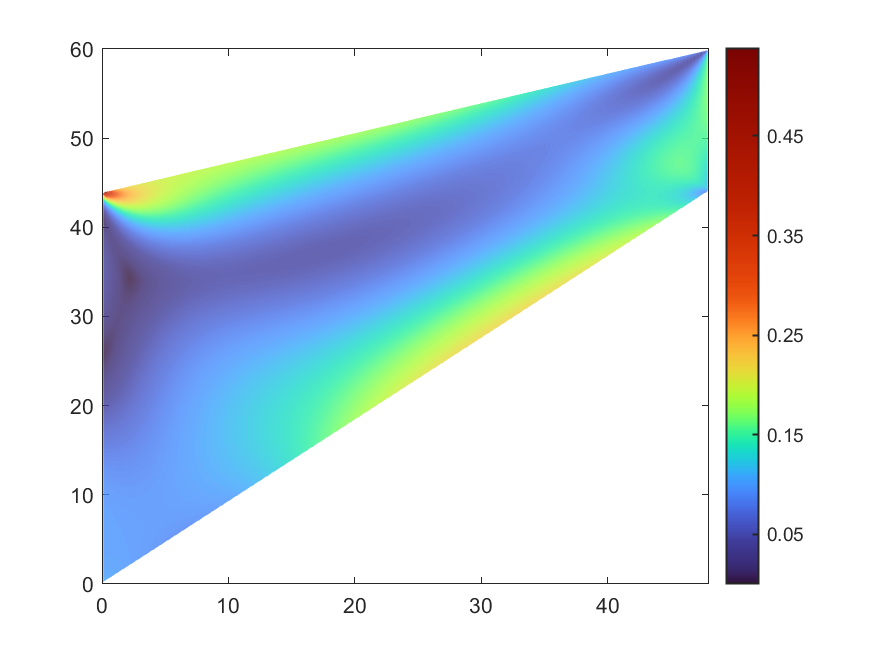}
 }\caption{Example \ref{Cook}: The numerical Von Misses stress in compressible and nearly incompressible cases on mesh of level 5.}
 \label{figsig}
\end{figure}

\end{example}

\section{Conclusions}\label{conclusion}
In this paper, we introduced a novel PF\&LF-EG method for linear elasticity problems in both 2D and 3D. Compared to the existing first-order EG method, our approach requires slightly more DoFs, as the enriched DG space is defined on edges or faces rather than on elements. However, this enrichment ensures an oscillation-free stress approximation without the need for post-processing. Additionally, the first-order divergence-free Stokes velocity element in \cite{hu2022family} can also lead to a locking-free method with an oscillation-free stress approximation. Compared to this method, our method maintains the same number of DoFs while avoiding the complexity of constructing modified bubbles with macro-element structures, reducing computational costs and improving feasibility. 
Furthermore, we establish rigorous error estimates that are independent of $\lambda$, theoretically confirming the method's locking-free property. Future work includes extending the method to higher-order approximations and applying it to other model problems.

\section*{Acknowledgements}
This work was partially  supported by the National Natural Science Foundation of China (No. 12201020).
\bibliographystyle{plain}
\bibliography{sample}

\end{document}